\documentclass[11pt,reqno,a4paper]{amsart}

\usepackage{amsmath,amssymb,amsthm,amsaddr}
\usepackage{mathtools}
\usepackage{enumitem}
\usepackage[margin=3cm,top=4cm,bottom=4cm,footskip=1cm,headsep=2cm]{geometry}
\usepackage[utf8]{inputenc}
\usepackage{float}

\usepackage{amsthm}
\usepackage[british]{babel}

\def\softd{{\leavevmode\setbox1=\hbox{d}%
		\hbox to 1.05\wd1{d\kern-0.4ex{\char039}\hss}}}

\usepackage{color}
\author{A. Brunk$^*$, H. Egger$^\dagger$, O. Habrich$^\dagger$, and M. Luk\' a\v cov\' a-Medvi\v{d}ov\'a$^*$}

\title[Relative energy estimates for the Cahn-Hilliard equation]{Relative energy estimates for the Cahn-Hilliard equation with concentration dependent mobility}

\def\div{\operatorname{div}}
\def\NN{\mathbb{N}}
\def\RR{\mathbb{R}}
\def\WW{\mathbb{W}}
\def\QQ{\mathbb{Q}}

\def\Th{\mathcal{T}_h}
\def\Itau{\mathcal{I}_\tau}
\def\I{I}

\def\la{\langle}
\def\ra{\rangle}

\DeclarePairedDelimiter{\norm}{\|}{\|}
\DeclarePairedDelimiter{\snorm}{|}{|}
\newcommand{\na}{\nabla}
\def\E{\mathcal{E}}
\def\D{\mathcal{D}}

\newtheorem{lemma}{Lemma}
\newtheorem{problem}[lemma]{Problem}
\newtheorem{theorem}[lemma]{Theorem}
\theoremstyle{definition}
\newtheorem{remark}[lemma]{Remark}

\def\dt{\partial_t}
\def\dtt{\partial_{tt}}

\def\ddt{\frac{d}{dt}}

\def\bmu{\bar{\mu}}
\def\hbmu{\hat{\bar{\mu}}}

\def\Vh{\mathcal{V}_h}

\usepackage{xcolor}

\begin{document}

\maketitle

	\centerline{$^*$ Institute of Mathematics, Johannes Gutenberg-University Mainz}
	
	\centerline{Staudingerweg 9, 55128 Mainz, Germany}
	
	\centerline{abrunk@uni-mainz.de, \quad lukacova@uni-mainz.de}

	\bigskip
	
	\centerline{$^\dagger$ Department of Mathematics, TU Darmstadt}
	
	\centerline{Dolivostraße 15, 64293 Darmstadt, Germany}

	\centerline{egger@mathematik.tu-darmstadt.de, \quad habrich@mathematik.tu-darmstadt.de}

\begin{abstract}
Based on relative energy estimates, we study the stability of solutions to the Cahn-Hilliard equation with concentration dependent mobility with respect to perturbations.
As a by-product of our analysis, we obtain a weak-strong uniqueness principle on the continuous level under realistic regularity assumptions on strong solutions.
We then show that the stability estimates  can be further inherited almost verbatim by appropriate Galerkin approximations in space and time. This allows us to derive sharp bounds for the discretization error in terms of certain projection errors and to establish order-optimal a-priori error estimates for semi- and fully discrete approximation schemes.
\end{abstract}

\section{Motivation}

The Cahn-Hilliard equation is one of the main mathematical models for describing phase separation phenomena, e.g., in binary alloys \cite{Cahn61,CahnHilliard58} or spinodal decomposition of binary fluids \cite{Boyer13}.
We study a system with concentration dependent mobility, given by 
\begin{alignat}{5} 
\dt \phi &= \div(b(\phi) \nabla \mu) \qquad && \text{in } \Omega, \ t>0, \label{eq:ch1}\\
\mu &= -\gamma \Delta \phi + f'(\phi) \qquad && \text{in } \Omega, \ t>0, \label{eq:ch2}
\end{alignat}
and complemented by appropriate initial and boundary conditions. 
As usual $\phi$ denotes the phase fraction, $\mu$ the chemical potential, $b(\phi)$ the concentration dependent mobility, $\gamma>0$ a parameter related to the interface thickness, and $f(\phi)$ is a double well potential whose minima characterize the two phases. 
The second equation defines the chemical potential $\mu = \delta_\phi \E(\phi)$ as the variational derivative of an associate energy 
\begin{align} \label{eq:energy}
\E(\phi) = \int_\Omega \frac{\gamma}{2} |\nabla  \phi|^2 + f(\phi),
\end{align}
which together with \eqref{eq:ch1} induces a gradient flow structure of the problem and decay of the energy $\E(\phi)$ along weak solutions. 
This implies thermodynamic consistency of the model and allows to establish existence of weak solutions by Galerkin approximation, energy estimates, and compactness arguments.

In \cite{ElliottFrench86}, existence and regularity of weak solutions for the Cahn-Hilliard equation has been established for constant mobility and polynomial potential. Logarithmic potentials and concentration dependent mobilities were treated in \cite{BarrettBloweyGarcke99,CopettiElliott92}. We refer to \cite{BarrettBlowey97,BarrettBloweyGarcke01,Feng06} for results concerning the extension to multi-component systems and multiphysical problems. See also our recent works where logarithmic potentials and concentration dependent mobility functions have been used in the context of a complex model for viscoelastic phase separation 
\cite{brunk, lukacova, strasser}.
Let us note that approximations in space and/or time and energy estimates are typically used to establish existence of solutions in rather general cases. 

Finite element approximations of the fourth-order system resulting after elimination of the chemical potential were analyzed in \cite{ElliottFrench86}. 
A mixed finite element approximation for constant mobilities treating $\phi$ and $\mu$ as separate variables was proposed in \cite{ElliottFrenchMilner89} and further analyzed in \cite{DuNicolaides91,ElliottLarsson92}. For extensions to logarithmic potentials and degenerate mobilities, we again refer to \cite{CopettiElliott92,BarrettBloweyGarcke99}. 
In \cite{FengProhl04,FengProhl05}, the analysis of finite element approximations has been extended to study the thin-interface limit $\gamma \to 0$.  
Apart from finite element methods, alternative discretization schemes, like discontinuous Galerkin methods \cite{KayStylesSuli09,LiuFrankRiviere19,XiaXuShu07} 
and Fourier-spectral approximations \cite{LiQiao17} have been investigated as well.
Extensive research has further been devoted to developing stable second order approximations in time; see \cite{TierraGuillenGonzales15} for an extensive overview and comparison of different approaches.
In a recent paper \cite{DiegelWangWise16}, which is probably closest to our investigations, an unconditionally well-posed fully discrete two-step approximation was proposed and a full convergence rate analysis was presented yielding order optimal error estimates. 

Let us note that quantitative convergence results in the papers cited above were derived mainly for the case of constant mobility, which allows to apply arguments of linear theory and to cover the terms stemming from the nonlinearity of the chemical potential $f$ by perturbation arguments. 
In this paper, we consider problems with \emph{concentration dependent mobility} and we utilize \emph{relative energy estimates}, i.e., genuinely nonlinear arguments, to conduct a quantitative error analysis. 
For ease of presentation, we will focus on conforming finite element approximations of \emph{second order in space and time}, but our arguments, in principle, apply also to higher order approximations and inexact Galerkin approximations. 

Related entropy and relative entropy arguments have been utilized intensively for the analysis of nonlinear evolution problems and, more recently, also for the convergence analysis of corresponding discretization methods. We refer to \cite{Juengel16} for an introduction and some selected results in this direction, as well as to \cite{doi:10.1137/16M1094233,gallouet:hal-01108579} for convergence and asymptotic analysis for fluid flow problems via relative energy estimates.
%

The first basic result of our paper is a formal relative entropy estimate which allows to deduce quantitative perturbation bounds for sufficiently regular solutions of \eqref{eq:ch1}--\eqref{eq:ch2}. As a by-product of our analysis, we also obtain a weak-strong uniqueness principle and thus a rather general proof of uniqueness.
Due to the variational character, these stability estimates are inherited almost verbatim by Galerkin approximations in space and Petrov-Galerkin approximation in time, which is our basic approach towards a systematic error analysis.
The structure of the relative energy estimates further provides guidelines for the choice of appropriate projection operators required in the error analysis. The discrete relative energy estimates then allow to estimate the discretization error by more or less standard projection error estimates, which finally leads to optimal convergence rates under minimal and less restrictive smoothness requirements than in previous works. This nonlinear convergence rate analysis can be seen as the main contribution of our manuscript.
%

In the current paper, we study problems with \emph{non-degenerate} concentration dependent mobility $b(\phi)$ and \emph{polynomially bounded} potential $f(\phi)$. These assumptions are used, for instance, to relate the relative entropy with the norm difference of solutions. 
\smallskip 

The remainder of the paper is organized as follows: 
In Section~\ref{sec:2}, we introduce our notation and basic assumptions and recall some results about existence and regularity of solutions. 
In Section~\ref{sec:3}, we introduce the relative entropy functional and present a formal relative entropy estimate which serves as the basis for the following considerations. Furthermore, we will also deduce the weak-strong uniqueness principle in course of the analysis.
In Section~\ref{sec:4}, we study the semi-discretization in space by a mixed finite element method. We will see that the relative energy estimate translates almost verbatim to the semi-discrete setting. This allows us to estimate the difference between the semi-discrete solution and a particular projection of the continuous solution by projection errors and to derive order optimal error convergence rates.
In Section~\ref{sec:5}, we then consider the time discretization by a Petrov-Galerkin approximation, 
which allows us to extend our arguments almost verbatim to the fully discrete setting. 
For illustration of our theoretical results, we present some preliminary numerical results in Section~\ref{sec:6}.
In the appendix we recall a version of the continuous and discrete Gronwall Lemma,  moreover we present the limiting process for the stability estimate from Section \ref{sec:3} and the higher regularity result for the weak solution are verified.

\section{Notation and preliminary results} \label{sec:2}

Let $L^p(\Omega)$, $W^{k,p}(\Omega)$ denote the usual Lebesgue and Sobolev spaces and $\norm{\cdot}_{0,p}$, $\norm{\cdot}_{k,p}$ the corresponding norms. In the Hilbert space case $p=2$, we write $H^k(\Omega)=W^{k,2}(\Omega)$ and abbreviate $\norm{\cdot}_k = \norm{\cdot}_{k,2}$. 
For ease of presentation, we will consider a periodic setting in the rest of the paper, and assume that $\Omega \subset \RR^d$ is a hyper cube in dimension $d=2,3$. 
We then write $H^s_p(\Omega)$, $s \ge 0$, for the space of functions in $H^s(\Omega)$ that can be extended periodically under preservation of class. 
The corresponding dual spaces are denoted by $H^{-s}_p(\Omega)=H^s_p(\Omega)'$. 
Note that for $s=0$, we have $H^s_p(\Omega) = H^{-s}_p(\Omega)=L^2_p(\Omega)$, where we identified $L^2(\Omega)$ with its dual space. 
The norm of the dual spaces are given by 
\begin{align} \label{eq:dualnorm}
    \norm{r}_{-s} = \sup_{v \in H^s_p(\Omega)} \frac{\la r, v\ra}{\|v\|_{s}},
\end{align}
where $\langle \cdot, \cdot\rangle$ denotes the duality product on $H^{-s}_p(\Omega) \times H^s_p(\Omega)$ for any $s \ge 0$.  
Note that for functions $u,v \in H^0_p(\Omega) = L^2_p(\Omega)$, we simply have $\la u,v \ra = \int_\Omega u v \, dx$, i.e., for sufficiently regular functions, the duality product can be identified with the scalar product of $L^2(\Omega)$.
If the meaning is clear from the context, we will sometimes omit the symbol $\Omega$ and briefly write $L^p$ for $L^p(\Omega)$, an so on.
We further denote by $L^p(a,b;X)$, $W^{k,p}(a,b;X)$,  and
$H^k(a,b;X)$ the Bochner spaces of correspondingly integrable or differentiable functions on the time interval $(a,b)$ with values in some Banach space $X$. If $(a,b)=(0,T)$, we will omit reference to the time interval and briefly write $L^p(X)$, for instance.

By a periodic weak solution of \eqref{eq:ch1}--\eqref{eq:ch2} on the interval $(0,T)$, we mean a pair of functions
\begin{align}
\phi &\in L^2(0,T;H^3_p(\Omega)) \cap H^1(0,T;H^1_p(\Omega)') =: \WW(0,T) \label{eq:weak_reg_1}\\
\mu &\in L^2(0,T;H^1_p(\Omega)) =: \QQ(0,T) \label{eq:weak_reg_2}
\end{align}
satisfying 
the variational identities
\begin{align}
 \la \dt \phi(t), v\ra + \la b(\phi(t)) \nabla \mu(t), \nabla v\ra &= 0, \label{eq:weak1}
\\
 \la \mu(t), w\ra - \la \gamma \nabla \phi(t), \nabla w\ra - \la f'(\phi(t)),w\ra &= 0,\label{eq:weak2}
\end{align} 
for all test functions $v,w \in H_p^1(\Omega)$ and a.a. $0<t<T$. 
It is not difficult to see that these two identities characterize all sufficiently regular periodic solutions of \eqref{eq:ch1}--\eqref{eq:ch2}.

In the rest of the paper, we make the following assumptions on the model parameters.
\begin{itemize}
\item[(A1)] $\gamma>0$ is a positive constant;
\item[(A2)] $b: \RR \to \RR_+$ satisfies $b\in C^2(\mathbb{R})$ with $0< b_1\leq b(s) \leq b_2$, $\norm{b'}_\infty\leq b_3$, $\norm{b''}_\infty\leq b_4$;
\item[(A3)] $f\in C^4(\mathbb{R})$ such that $f(s),f''(s)\geq -f_1$, for $f_1\geq0$. Furthermore, we assume 
that $f$ and its derivatives are bounded by $|f^{(k)}(s)| \le f_2^{(k)} + f_3^{(k)} |s|^{4-k}$ for $0 \le k \le 4$.
\end{itemize}
The growth bounds for $f$ immediately imply that $f(\phi)\in L^1(\Omega)$ for every $\phi\in H^1(\Omega)$. Functions $\phi \in H^1_p(\Omega)$ therefore have bounded energy $\E(\phi) < \infty$.

Under these assumptions, the existence of periodic weak solutions can be deduced from classical results. For later reference, we make a corresponding statement. 
\begin{lemma} \label{lem:weak}
Let (A1)--(A3) hold. Then for any $\phi_0 \in H^1_p(\Omega)$, 
there exists at least one periodic weak solution $(\phi,\mu)$ of problem \eqref{eq:ch1}--\eqref{eq:ch2} with initial value $\phi(0)=\phi_0$,  
and any periodic weak weak solution $(\phi,\mu)$ satisfies
\begin{align*}
\int_\Omega \phi(t) dx = \int_\Omega \phi_0 \, dx \qquad \text{and} \qquad
\E(\phi(t)) \le \E(\phi_0) - \int_0^t \D_{\phi(s)}(\mu(s)) \, ds, 
\end{align*}
for a.a. $0 \le t \le T$ with $\D_\phi(\mu)= \|b^{1/2}(\phi)\na\mu\|_0^2$ denoting the \emph{dissipation functional}.\\
If $\phi_0 \in H_p^k(\Omega)$, $1 \le k \le 3$, and $T$ sufficiently small for $d=3$ and $k > 1$, we further have
\begin{align*}
    \|\phi\|_{L^\infty(H_p^k)} + \|\phi\|_{L^2(H_p^{k+2})} + \|\dt \phi\|_{L^2(H_p^{k-2})} + \|\mu\|_{L^2(H^k_p)} + \|\mu\|_{L^\infty(H_p^{k-2})} \le C_{T}(\|\phi_0\|_N) 
\end{align*}
with constant $C_{T}(\|\phi_0\|_N)$ depending on the bounds for the coefficients and the domain.  
\end{lemma}
\begin{proof}
Existence of weak solutions and the a-priori bounds for $k=1$ and any $T>0$ in dimension $d=2,3$ follow from standard arguments; see  \cite{BarrettBlowey99,BarrettBloweyGarcke99b} for similar results under even more general assumptions on the problem data.
Conservation of mass and dissipation of energy follow immediately from the variational identities \eqref{eq:weak1}--\eqref{eq:weak2} by formally testing with $(v,w)=(1,0)$ and $(v,w)=(\mu,\dt \phi)$, respectively.
Improved regularity and the bounds for the solution for $k > 1$, which require a restriction on the maximal time $T$ in dimension $d=3$, can be obtained by a boot-strap argument and regularity results for the Poisson problem; details are given in the appendix. 
\end{proof}
\begin{remark}
From the estimates of Lemma~\ref{lem:weak} and the embedding theorem for Bochner spaces, see e.g. \cite[Ch.~25]{Wloka}, one can see that weak solutions $\phi$ and $\mu$ are continuous in appropriate function spaces. For regular initial values $\phi_0 \in H^3_p(\Omega)$, for instance, one has
\begin{align} \label{eq:continuous}
(\phi,\mu) \in C([0,T];H_p^3(\Omega) \times H^1_p(\Omega)),    
\end{align}
and hence $\phi$ is uniformly bounded on $\Omega \times (0,T)$. This will be used in  Section~\ref{sec:4} below.
\end{remark}

\section{A stability estimate and uniqueness}\label{sec:3}

As a first step of our analysis, we study the stability of periodic weak solutions $(\phi,\mu)$ of the system \eqref{eq:ch1}--\eqref{eq:ch2} with respect to perturbations.
Let $(\hat \phi,\hat \mu)$ be a pair of sufficiently regular functions.
Then the variational identities
\begin{align}
 \la \dt \hat \phi(t), v \ra + \la b(\phi(t)) \nabla \hat \mu(t), \nabla v \ra &=: \la \hat r_1(t),v \ra, \label{eq:weak1p}
\\
 \la \hat \mu(t), w\ra - \la \gamma \nabla \hat \phi(t), \nabla w\ra - \la f'(\hat \phi(t)),w\ra &=: \la \hat r_2(t),w\ra, \label{eq:weak2p}
\end{align}
for all $v,w \in H^1_p(\Omega)$ and a.a. $0 < t < T$, define two residual functionals $\hat r_1(t),\hat r_2(t)$.
Alternatively, the functions $(\hat \phi, \hat \mu)$ can be understood as solutions of the perturbed variational problem \eqref{eq:weak1p}--\eqref{eq:weak2p} for given right hand side $\hat r_1,\hat r_2$.
\begin{remark}
Let us emphasize that a term $b(\phi)$ appears in \eqref{eq:weak1p} which explicitly depends on the solution $\phi$ of \eqref{eq:ch1}--\eqref{eq:ch2}. Equation \eqref{eq:weak1p} therefore includes some sort of linearization around $\phi$ which greatly simplifies the proofs of our further results.  
\end{remark}

\subsection{Stability via relative energy}

In order to measure the difference between a given solution $(\phi,\mu)$ of \eqref{eq:weak1}--\eqref{eq:weak2} and solution $(\hat \phi,\hat \mu)$ of the perturbed problem \eqref{eq:weak1p}--\eqref{eq:weak2p}, we will utilize a regularized \emph{relative energy} functional
\begin{align} \label{eq:relenergy}
\E_\alpha(\phi|\hat \phi) 
:= \E(\phi) - \E(\hat \phi) - \la \E'(\hat \phi), \phi - \hat \phi\ra  + \tfrac{\alpha}{2} \|\phi - \hat \phi\|^2,
\end{align}
for some $\alpha>0$ chosen such that the regularized energy functional $\E_\alpha(\phi|\hat\phi)=\E(\phi) + \frac{\alpha}{2} \|\phi\|_0^2$ becomes strictly convex.
This can be achieved, e.g., by choosing
\begin{itemize}
    \item[(A4)] $\alpha = \max\{\gamma,\gamma+f_1\}$, 
\end{itemize}
where $-f_1$ is the constant in the lower bound for $f''$ from assumption (A3). 
The relative energy functional $\E_\alpha(\phi|\hat \phi)$ then is the associated Bregman distance \cite{Bregman67}. 
Moreover, the norm distance of two functions can be bounded by the relative energy.
\begin{lemma} \label{lem:equiv}
Let (A1)--(A4) hold. Then
\begin{align} \label{eq:lower_bound_rel}
\frac{\gamma}{2} \|\phi - \hat \phi\|_{1}^2  \le \E_\alpha(\phi|\hat \phi) \le \Gamma (1+\|\phi\|_1^2+\|\hat \phi\|_1^2) \|\phi - \hat \phi\|_1^2,
\end{align}
for all functions $\phi,\hat \phi \in H^1_p(\Omega)$ with uniform constant $\Gamma=\Gamma(f_2^{(2)}, f_3^{(2)},\Omega)$.
\end{lemma}
\begin{proof}
For $g(x)=|x|^2$, one has $g(x|\hat x) = g(x) - g(\hat x) - \la g'(\hat x), (x-\hat x) \rangle = |x-\hat x|^2$ which allows to handle the quadratic contributions. It thus suffices to consider the nonlinear terms in $\E_\alpha(\phi|\hat \phi)$.  
The lower bound then follows directly from noting that 
\begin{align*}
\int_\Omega f(\phi|\hat \phi) dx  \ge -\frac{f_1}{2} \|\phi - \hat \phi\|_0^2
\end{align*}
and the particular choice of $\alpha$.
For the upper bound, we use the growth bounds for $f$, which allow us to show that
\begin{align*}
\int_\Omega f(\phi|\hat \phi) dx 
&= \int_\Omega \int_0^1 f''(s \phi + (1-s) \hat \phi) ds \, (\phi - \hat \phi)^2 dx \\
&\le (f_2^{(2)} + f_3^{(2)} (|\phi|^2 + |\hat \phi|^2)) |\phi - \hat \phi|^2 dx \\
&\le f_2^{(2)} \|\phi - \hat \phi\|^2_0 + f_3^{(2)} (\|\phi\|_{0,4}^2 + \|\hat \phi\|_{0,4}^2) \|\phi - \hat \phi\|_{0,4}^2.
\end{align*}
In the last step, we simply used H\"older's inequality, and by embedding of $H^1$ into $L^p$, we may estimate $\|\phi\|_{0,4} \le C(\Omega) \|\phi\|_1$, which yields the required upper bound. 
\end{proof}
%

%
Using the specific problem structure and elementary computations, we can now derive the following stability estimate which will be the basis for our further considerations.
\begin{theorem} \label{thm:main}
Let (A1)--(A4) hold and $(\phi,\mu) \in \WW(0,T) \times \QQ(0,T)$ denote a periodic weak solution of \eqref{eq:ch1}--\eqref{eq:ch2}. 
Furthermore, let $\hat \phi \in \WW(0,T) \cap W^{1,1}(0,T;L^2(\Omega))$ 
and $\hat \mu \in \QQ(0,T)$ be given and $(\hat r_1,\hat r_2)$ be defined by \eqref{eq:weak1p}--\eqref{eq:weak2p}. Then
\begin{align} \label{eq:stability}
  \E_\alpha(\phi(t)|\hat \phi(t)) 
   &+ \int_0^t \D_{\phi(s)}(\mu(s)|\hat \mu(s)) \, ds  \\
&\le e^{c (t)} \E_\alpha(\phi(0)|\hat \phi(0)) + C e^{c(t)}\int_0^t \|\hat r_1(s)\|^2_{-1} + \|\hat r_2(s)\|^2_{1} \, ds \notag
\end{align}
with relative dissipation functional $\D_{\phi}(\mu|\hat \mu) = \frac{1}{2} \|b^{1/2}(\phi) \nabla (\mu - \hat \mu)\|^2_0$, parameter $c(t)=c_0 t +c_1 \int_0^t \|\dt \hat\phi(s)\|_0 ds$, and constants $c_0,c_1,C$ depending only on the domain $\Omega$ and the uniform bounds for the functions $(\phi,\mu)$ and $(\hat \phi,\hat \mu)$ in $L^\infty(H^1) \times L^2(L^2)$.
\end{theorem}
\begin{proof}
For ease of presentation, we assume for the moment that $(\phi,\mu)$ and $(\hat \phi,\hat \mu)$ are sufficiently regular, such that all computations in the following are justified. The general case can then be deduced by a density argument; details are given in the appendix.  
By formal differentiation of the relative energy with respect to time, we get
\begin{align*}
\ddt \E_\alpha(\phi|\hat \phi) 
&= \la \E_\alpha'(\phi), \dt \phi\ra - \la \E_\alpha'(\hat \phi), \dt \hat \phi\ra -  
    \la \E_\alpha'(\hat \phi), \dt \phi - \dt \hat \phi \ra - \la \E_\alpha''(\hat \phi) \dt \hat \phi, \phi - \hat \phi\ra \\
&= \la \E_\alpha'(\phi) - \E_\alpha'(\hat \phi), \dt \phi - \dt \hat \phi \ra + 
   \la \E_\alpha'(\phi) - \E_\alpha'(\hat \phi) - \E_\alpha''(\hat \phi) (\phi-\hat \phi), \dt \hat \phi  \ra. 
\end{align*}
Inserting the definition of the relative energy $\E_\alpha$ and using the variational identities \eqref{eq:weak1}--\eqref{eq:weak2} and \eqref{eq:weak1p}--\eqref{eq:weak2p}, which are satisfied by the functions $(\phi,\mu)$ and $(\hat \phi, \hat \mu)$,  
we obtain
\begin{align*}
\ddt \E_\alpha(\phi|\hat \phi) 
&= \gamma \la \nabla \phi - \nabla \hat \phi, \nabla \dt \phi - \nabla \dt \hat \phi\ra 
    +\la f'(\phi) - f'(\hat\phi), \dt\phi - \dt\hat\phi \ra \\
    &\qquad \qquad  + \alpha \la \phi - \hat \phi, \dt \phi - \dt \hat \phi\ra 
   + \la f'(\phi) - f'(\hat \phi) - f''(\hat \phi) (\phi - \hat \phi), \dt \hat \phi\ra \\
&= \la \mu - \hat \mu + \hat r_2, \dt \phi - \dt \hat \phi\ra 
   + \alpha \la \phi- \hat \phi, \dt \phi - \dt \hat \phi\ra 
\\ &\qquad \qquad 
   + \la f'(\phi) - f'(\hat \phi) - f''(\hat \phi) (\phi - \hat \phi), \dt \hat \phi\ra \\
&=  -\la b(\phi) \nabla (\mu - \hat \mu), \nabla(\mu-\hat \mu + \hat r_2)\ra - \la r_1, \mu - \hat \mu + \hat r_2 \ra \\
& \qquad \qquad - \alpha \la b(\phi) \nabla (\mu - \hat \mu), \nabla (\phi - \hat \phi) \ra - \alpha \la r_1, \phi - \hat \phi \ra \\
& \qquad \qquad + \la f'(\phi) - f'(\hat \phi) - f''(\hat \phi)(\phi - \hat \phi), \dt \hat \phi \ra \\
&= (i) + (ii) + (iii) + (iv) + (v). 
\end{align*}
In what follows, we will estimate the individual terms of the last line separately.
Before we proceed, let us note that by the energy bounds for weak solutions $(\phi,\mu)$, see Lemma~\ref{lem:weak}, and by the assumptions on $\hat \phi$ in the statement of the Lemma, we know that
\begin{align} \label{eq:bound}
    \|\phi\|_{L^\infty(H^1)} \le C(\|\phi_0\|_1) 
    \qquad \text{and} \quad 
    \|\hat \phi\|_{L^\infty(H^1)}, \|\dt \hat \phi\|_{L^1(L^2)} \le \hat C. 
\end{align}
Using Hölder's and Young's inequalities, we can then bound 
\begin{align*}
(i) &= -\norm{b^{1/2}(\phi)\na(\mu-\hat\mu)}_0^2 + \la b(\phi)\na(\mu-\hat\mu), \hat r_2\ra \\
&\leq -(2-2\delta) \D_\phi(\mu|\hat \mu) + C(\delta,b_2)\norm{\hat r_2}_1^2,    
\end{align*}
with $\delta>0$ arbitrary, constant $C(\delta,b_2) = b_2/(4\delta)$, and $b_2$ denoting the upper bound for the function $b$ in assumption (A2). 
By definition of the dual norm, a Poincar\'e inequality, and the bounds for the coefficients, the second term can be further estimated by
\begin{align*}
(ii) &\leq \norm{\hat r_1}_{-1}\left(\norm{\mu-\hat\mu}_1+\norm{\hat r_2}_1\right) \\
&\leq \norm{\hat r_1}_{-1}\left( C(\Omega) |\la \mu-\hat\mu,1 \ra| + C(\Omega,b_1)\norm{b^{1/2}(\phi)\na(\mu-\hat\mu)}_2+\norm{\hat r_2}_1\right)  \\
&\leq  C(\Omega,b_1,\delta) \norm{\hat r_1}_{-1}^2 + |\la \mu-\hat\mu,1 \ra|^2 + 2 \delta \, D_\phi(\mu|\hat\mu) +  \norm{\hat r_2}_1^2.
\end{align*}
In the last step, we utilized Youngs' inequality to separate the factors with the same arbitrary parameter $\delta>0$ as before.
For the second term on the right hand side, we can use the variational identities \eqref{eq:weak2} and \eqref{eq:weak2p} with $w=1$, which leads to
\begin{align*}
    |\la \mu-\hat\mu,1 \ra| 
    &= |\la f'(\phi) -f'(\hat\phi) + \hat r_2,1\ra| \\
    &\leq C(\Omega) (\norm{\hat r_2}_{0,1} + \|f'(\phi)-f'(\hat\phi)\|_{0,1}).
\end{align*}
From the bounds for the potential $f$ in assumption (A3), we can further deduce that 
\begin{align*}
|f'(\phi) - f'(\hat \phi)| 
& = \left|\int_0^1 f''(\hat\phi + s(\phi - \hat \phi)) ds (\phi - \hat \phi) \right| \\
&\le \left(f_2^{(2)} +  f_3^{(2)} (|\phi| + |\hat \phi|)^2\right) |\phi - \hat \phi|.
\end{align*}
An application of H\"older's inequality, the norm estimates for the continuous embedding of $H^1(\Omega)$ into $L^p(\Omega)$, and the uniform bounds for $\phi$, $\hat \phi$ in \eqref{eq:bound}, then lead to
\begin{align*}
\|f'(\phi) - f'(\hat \phi)\|_{0,1} 
&\le  \left( C(\Omega) f_2^{(2)}  + 2 f_3^{(2)} (\|\phi\|_{0,6}^2 + \|\hat \phi\|_{0,6}^2) \right) \|\phi - \hat \phi\|_{0,6} \\
&\le  C(\Omega,f_2^{(2)},f_{3}^{(2)},\|\phi_0\|_1,\|\hat \phi\|_{L^\infty(H^1)}) \|\phi - \hat \phi\|_1.
\end{align*}
Using $\|\hat r_2\|_{0,1} \le C(\Omega)\|\hat r_2\|_1$ and the lower bound \eqref{eq:lower_bound_rel} for the relative energy, we arrive at
\begin{align*}
(ii) 
\le 2\delta \, \D_\phi(\mu|\hat \mu) &+ C(\Omega,b_1,\delta) \norm{r_1}_{-1}^2 + C(\Omega)\norm{\hat r_2}_1^2 \\
&+ C(\Omega,f_2^{(2)},f_3^{(2)},\|\phi_0\|_1, \|\hat \phi\|_{L^\infty(H^1)},\gamma)\,  \E_\alpha(\phi|\hat \phi).
\end{align*}
Condition \eqref{eq:lower_bound_rel} further allows us to estimate 
\begin{align*}
(iii) + (iv) 
&\leq 2\delta \, \D_\phi(\mu|\bar \mu) + C(\delta,b_2,\alpha,\gamma) \E_\alpha(\phi|\hat \phi) + C(\alpha) \norm{r_1}_{-1}^2 .  
\end{align*}
From the bounds in assumption (A3), we can deduce that 
\begin{align*}
|f'(\phi) - f'(\hat \phi) - f''(\hat \phi)(\phi - \hat \phi)| 
&\le \left(f_2^{(3)} + f_3^{(3)}(|\phi| + |\hat \phi|) \right) |\phi - \hat \phi|^2.
\end{align*}
Using H\"olders inequality, embedding estimates, and the uniform bounds in \eqref{eq:bound}, we can further bound the fifth term in the above estimate by 
\begin{align*}
(v) &\le \|\dt \hat \phi\|_{0} \|f'(\phi) - f'(\hat \phi) - f''(\hat \phi)(\phi - \hat \phi)\|_{0} \\
&\le \|\dt \hat \phi\|_{0} \left( f_2^{(3)} + f_3^{(3)} (\|\phi\|_{0,6} + \|\hat \phi\|_{0,6}) \right) \|\phi - \hat \phi)\|_{0,6}^2 \\
&\le C(\Omega,f_2^{(3)},f_3^{(3)},\|\phi_0\|_1,\|\hat \phi\|_{L^\infty(H^1)},\gamma) \|\dt \hat \phi\|_0 \, \E_\alpha(\phi|\hat \phi).
\end{align*}
By combination of the individual estimates and choosing $\delta = 1/6$, we finally obtain 
\begin{align*}
\ddt \E_\alpha(\phi|\hat \phi) 
&\le -\D_\phi(\mu|\hat\mu) + (c_0 + c_1 \|\dt \hat \phi\|_0) \E_\alpha(\phi|\hat \phi) +  
C_2 \|\hat r_1\|_{-1}^2  + C_3 \|\hat r_2\|_{1}^2,
\end{align*}
with constants $c_0,c_1$, $C_2$, $C_3$ depending only on the bounds for the coefficients, the domain, and the bounds for $\|\phi_0\|_1$ and $\|\hat \phi\|_{L^\infty(H^1)}$.  
An application of Gronwall's inequality \eqref{eq:gronwall} with $v(t)=\E_\alpha(\phi(t)|\hat \phi(t))$, $g(t)=-\D_{\phi(t)} + C_2 \|\hat r_1(t)\|_{-1} + C_3 \|\hat r_2(t)\|_{1}$, and $\lambda(t)=c_0 + c_1 \|\dt \hat \phi(t)\|_0$,
which is integrable since $\dt \hat \phi \in L^1(L^2)$, then leads to the stability estimate of the theorem with constants $c=c_0 T + c_1 \hat C$ and $C=\max\{C_2,C_3\}$. 
\end{proof}

\begin{remark}
The lower bound \eqref{eq:lower_bound_rel} for the relative energy, and the bound
\begin{align}
\frac{b_1}{2}\norm{\nabla\mu-\nabla\hat\mu}_{0}^2 \leq \D_\phi(\mu|\hat\mu), \label{eq:lower_bound_rel_dis}
\end{align}
for the relative dissipation immediately lead to uniform bounds 
\begin{align*}
\|\phi - \hat \phi\|_{L^\infty(H^1)}^2 + \|\mu - \hat \mu\|_{L^2(H^1)}^2 
\le C_1 \E_\alpha(\phi(0)|\hat\phi(0)) + C_2 (\|\hat r_1\|_{L^2(H^{-1}_p)}^2 + \|\hat r_2\|_{L^2(H^1)}^2) 
\end{align*}
for the error. With similar arguments as used for the estimate of the term (ii), we can also bound the full norm $\|\mu - \hat \mu\|_{L^2(H^1_p)}$.    
The stability estimate thus provides perturbation bounds in the the natural norms to be used for the error analysis of the problem.
\end{remark}

\subsection{A weak-strong uniqueness principle}

As a direct consequence of Theorem \ref{thm:main}, one can see that (sufficiently regular) weak solutions of \eqref{eq:ch1}--\eqref{eq:ch2} depend stably on perturbations in the problem parameters and the initial data. 
Another consequence of Theorem~\ref{thm:main} is the following weak-strong uniqueness principle. 
\begin{theorem} \label{thm:unique}
Let $(\hat \phi, \hat \mu)$ denote a periodic weak solution of \eqref{eq:ch1}--\eqref{eq:ch2} with improved regularity $\hat \phi \in W^{1,1}(0,T;L^2(\Omega))$ and $\hat \mu \in L^2(0,T;W_p^{1,3}(\Omega))$.
Then no other weak solution $(\phi,\mu)$ with the same initial values $\phi(0)=\hat \phi(0)$ can exist. 
\end{theorem}
\begin{proof}
Let $(\phi,\mu)$ be a weak solution with the same initial values $\phi(0)=\hat \phi(0)$. %
Then $(\hat \phi,\hat \mu)$ can be seen to solve \eqref{eq:weak1p}--\eqref{eq:weak2p} 
with residuals 
\begin{align*}
    \la \hat r_1,v\ra = \la (b(\phi)-b(\hat\phi))\na\hat\mu, \nabla v\ra 
    \qquad \text{and} \qquad \hat r_2=0,
\end{align*}
and the first residual can be further estimated by 
\begin{align*}
    \int_0^t\|\hat r_1(s)\|_{-1}^2 ds &\le \int_0^t\| (b(\phi(s))-b(\hat\phi(s)))\na\hat\mu(s) \|_0^2 ds\\
    &\leq C(b_3)\int_0^t\|\phi(s) -\hat \phi(s)\|_{0,6}^2\|\na\hat\mu(s)\|_{0,3}^2 ds\\
    &\leq C(b_3,\gamma,\Omega) \int_0^t\|\na\hat\mu(s)\|_{0,3}^2\E(\phi(s)|\hat\phi(s)) ds.
\end{align*}
By assumption on the initial values, we have $\E_\alpha(\phi(0)|\hat \phi(0))=0$, and the estimate of Theorem~\ref{thm:main} thus directly leads to 
\begin{align*}
\E_\alpha(\phi(t)|\hat \phi(t)) 
   + \int_0^t \D_{\phi(s)}(\mu(s)|\hat \mu(s)) \, ds 
   \le C(T) \int_0^t \|\na\hat\mu(s)\|_{0,3}^2\E(\phi(s)|\hat\phi(s)) ds.
\end{align*}
Since $\|\nabla \hat \mu(t)\|_{0,3}^2 \in L^1(0,T)$, we can use Gronwall's inequality \eqref{eq:gronwall} once more, leading to $\E_\alpha(\phi(t)|\hat \phi(t)) \le 0$ for $0 \le t \le T$, which together with Lemma~\ref{lem:equiv} yields the claim. 
\end{proof}

\begin{remark}
Note that for regular initial values, e.g., $\hat \phi(0)=\phi_0 \in H^2_p(\Omega) $, the existence of a weak solution $(\hat \phi, \hat \mu)$ with the required extra regularity follows from Lemma~\ref{lem:weak}. In that case, we therefore have a unique weak solution.
\end{remark}

\section{Galerkin semi-discretization} \label{sec:4}

We now turn to the discretization of \eqref{eq:weak1}--\eqref{eq:weak2} in space,
for which we consider a conforming Galerkin approximation of the variational principle \eqref{eq:weak1}--\eqref{eq:weak2} with \emph{second order} conforming finite elements.
As will become clear from our analysis, higher order and, to some extent, also non-conforming approximations could be treated with similar arguments.

Let $\Th$ denote geometrically conforming partition of $\Omega \subset \RR^d$, $d=2,3$ into triangles or tetrahedra. As usual, we denote by $\rho_N$ and $h_N$ the inner-circle radius and diameter of the element $K \in \Th$ and call $h=\max_{K \in \Th} h_T$ the global mesh size. We assume that $\Th$ is quasi-uniform, i.e., there exists a constant $\sigma>0$ such that $\sigma h \le \rho_N \le h_N \le h$ for all $K \in \Th$. 
We further assume that the mesh $\Th$ is periodic in the sense that it can be extended periodically to periodic extensions of the domain $\Omega$. 
We then denote by 
\begin{align*}
    \Vh := \{v \in H^1_p(\Omega) : v|_N \in P_2(K) \quad \forall K \in \Th\},
\end{align*}
the space of continuous periodic piecewise quadratic polynomials over the mesh $\Th$.  
We further introduce the approximation spaces
$$
\WW_h(0,T):= H^1(0,T;\Vh) \qquad \text{and} \qquad \QQ_h(0,T):= L^2(0,T;\Vh).
$$
The semi-discrete approximation for \eqref{eq:weak1}--\eqref{eq:weak2} then reads as follows. 
\begin{problem} \label{prob:semi}
Let $\phi_{0,h} \in \Vh$ be given. Find $(\phi_h,\mu_h) \in \WW_h(0,T)\times\QQ_h(0,T)$  such that $\phi_h(0)=\phi_{0,h}$ and
such that for all $vh, w_h \in \Vh$ and all $0 \le t \le T$, there holds
\begin{align}
 \la\dt \phi_h(t), v_h\ra + \la b(\phi_h(t)) \nabla \mu_h(t), \nabla v_h\ra &= 0, \label{eq:weak1h}
\\
 \la \mu_h(t), w_h\ra - \la\gamma \nabla \phi_h(t), \nabla w_h\ra - \la f'( \phi_h(t)),w_h\ra &= 0. \label{eq:weak2h}
\end{align}
\end{problem}
Before we turn to a detailed stability and error analysis, let us briefly summarize some basic properties of this discretization strategy. 
\begin{lemma}
\label{lem:energy}
Let (A1)--(A3) hold. Then for any initial value $\phi_{0,h} \in \Vh$, Problem~\ref{prob:semi} has a unique solution $(\phi_h,\mu_h)$. 
Moreover, for all $0 \le t \le T$, one has $\int_\Omega \phi_h(t) dx = \int_\Omega \phi_{0,h} dx$ as well as $\E(\phi_h(t))  + \int_0^t \D_{\phi_h}(\mu_h) ds = \E(\phi_{0,h})$.
\end{lemma}
\begin{proof}
Using (A1)--(A3), existence of a unique solution $(\phi_h,\mu_h) \in C^1(0,T;\Vh) \times C^0(0,T;\Vh)$ can be deduced from the Picard-Lindelöf theorem. The mass conservation and energy dissipation identities then follow with similar arguments as on the continuous level by testing equations \eqref{eq:weak1h}--\eqref{eq:weak2h} with $(v_h,w_h)=(1,0)$ and $(v_h,w_h)=(\mu_h,\dt \phi_h)$, respectively. 
\end{proof}

\subsection{Semi-discrete stability estimate}

With similar arguments as used on the continuous level, we will now establish stability of the semi-discrete solution with respect to perturbations. 
For a given pair of functions $(\hat \phi_h,\hat \mu_h)\in \WW_h(0,T)\times \QQ_h(0,T)$, we define semi-discrete residuals $(\hat r_{1,h},\hat r_{2,h}) \in L^2(0,T;\Vh \times \Vh)$ by the variational identities
\begin{align}
 \la\dt \hat \phi_h(t), v_h\ra + \la b(\phi_h(t)) \nabla \hat \mu_h(t), \nabla v_h\ra &=: \la \hat r_{1,h}(t),v_h \ra, \label{eq:weak1hp}
\\
 \la \hat \mu_h(t), w_h\ra - \la\gamma \nabla \hat \phi_h(t), \nabla w_h\ra - \la f'(\hat \phi_h(t)),w_h\ra &=: \la \hat r_{2,h}(t),w_h\ra,\label{eq:weak2hp}
\end{align}
for all $v_h, w_h \in \Vh$ and $0 \le t \le T$. 
The functions $(\hat\phi_h,\hat\mu_h)$ can again be understood as solutions of the perturbed semi-discrete problem \eqref{eq:weak1hp}--\eqref{eq:weak2hp}.
With almost identical arguments as used in the proof of Theorem~\ref{thm:main}, 
we now obtain the following stability estimate. 
\begin{lemma} \label{lem:semi}
Let (A1)--(A4) hold and $(\phi_h,\mu_h) \in \WW_h(0,T)\times \QQ_h(0,T)$ denote a solution of Problem~\ref{prob:semi}.
Furthermore, let 
$(\hat \phi_h,\hat \mu_h) \in \WW_h(0,T) \times \QQ_h(0,T)$ be given 
and $(\hat r_{1,h},\hat r_{2,h})$ denote the residuals defined by \eqref{eq:weak1hp}--\eqref{eq:weak2hp}.
Then the estimate
\begin{align} \label{eq:disc_stability}
  \E_\alpha(\phi_h(t)|\hat \phi_h(t)) &+ \int_0^t   \D_{\phi_h(s)}(\mu_h(s)|\hat\mu_h(s)) \, ds \\  &\le e^{c(t)} \E_\alpha(\phi_h(0)|\hat \phi_h(0)) + C e^{c(t)}\int_0^t \|\hat r_{1,h}(s)\|^2_{-1,h} + \|\hat  r_{2,h}(s)\|^2_{1} \, ds \notag
\end{align}
holds for a.a $0 \le t \le T$ with parameter $c(t)=c_0 t+c_1 \int_0^t \|\dt \hat \phi_h\|_0 ds$ and $c_0,c_1,C$ depending on the uniform $L^\infty(H^1) \times L^2(H^1)$ bounds for $(\phi_h,\mu_h)$ and $(\hat \phi_h, \hat \mu_h)$, respectively, and  

\begin{align} \label{eq:dualh}
\|\hat r\|_{-1,h} = \sup_{v_h \in \Vh} \frac{(\hat r,v_h)}{\|v_h\|_{1}} \le \|r\|_{-1} 
\end{align}
denoting the \emph{discrete-dual norm}. 
\end{lemma}
\begin{proof}
The assertion follows with the very same arguments as used in the proof of Theorem~\ref{thm:main}; the details are left to the reader.
\end{proof}
Lemma~\ref{lem:semi} allows to investigate the stability of the semi-discrete solution $(\phi_h,\mu_h)$ with respect to perturbations in the initial conditions and problem data.
We will choose $(\hat \phi_h,\hat \mu_h)$ as a particular discrete approximation for the solution $(\phi,\mu)$ of \eqref{eq:ch1}--\eqref{eq:ch2}. 
This will allow us to derive quantitative error estimates for the semi-discrete approximation. 

\subsection{Auxiliary results} 

We start by introducing some projection operators and recall the corresponding error estimates.
Let $\pi_h^0 : H^1_p(\Omega) \to \Vh$ denote the $L^2$-orthogonal projection which can be be characterized by
\begin{align} \label{eq:l2proj}
    \la\pi_h^0 u - u, v_h\ra = 0 \qquad \forall v_h \in \Vh.
\end{align}
By definition, $\pi_h^0$ is a contraction in $L^2(\Omega)$ and on quasi-uniform meshes, $\pi_h^0$ is also stable with respect to the $H^1$-norm, i.e., $\|\pi_h^0 u\|_1 \le C(\sigma) \|u\|_1$ for all $u \in H^1_p(\Omega)$; see \cite{BrennerScott}. 
Moreover
\begin{align} \label{eq:l2projest}
    \|u - \pi_h^0 u\|_{s} \leq C h^{r-s} \|u\|_{r}
\end{align}
for all $-1 \le s \le r$ and $0 \le r \le 3$. 
In our analysis, we will also utilize the $H^1$-elliptic projection $\pi_h^1 : H^1_p(\Omega) \to \Vh$, which is characterized by the variational problem 
\begin{align} \label{eq:hatphih}
\la \nabla ( \pi_h^1 u - u), \nabla v_h\ra + \la \pi_h^1 u -  u,v_h\ra &= 0 \qquad \forall v_h \in \Vh.
\end{align}
By standard finite element error analysis and duality arguments, one can show that
\begin{align}\label{eq:h1porjest}
 \|u - \pi_h^1 u\|_{s} \leq C h^{r-s} \|u\|_r,
\end{align}
for all $-1 \le s \le r$ and $1 \le r \le 3$; see again \cite{BrennerScott} for details.
Since we assumed quasi-uniformity of the mesh $\Th$, we can further resort to the inverse inequalities
\begin{align} \label{eq:inverse}
    \|v_h\|_1 \le c_{inv} h^{-1} \|v_h\|_0 
    \qquad \text{and} \qquad 
    \|v_h\|_{0,p} \le c_{inv} h^{d/p-d/q} \|v_h\|_{0,q} 
\end{align}
which hold for all discrete functions $v_h \in \Vh$ and all $1 \le q \le p \le \infty$ of a quasi-uniform simplicial mesh in dimension $d$. 
By combining the previous estimates, one can see that 
\begin{align} \label{eq:uniform}
    \|\pi_h^1 u\|_{0,\infty} \le C \|u\|_{2}
\end{align} 
in dimension $d \le 3$. 
Let us note that all  estimates also hold in dimension one, i.e., for piecewise polynomial approximations in time. 

\subsection{Projection error estimates}

Let $(\phi,\mu)$ be a periodic weak solution of \eqref{eq:ch1}--\eqref{eq:ch2}.
We then define $\hat \phi_h(t) = \pi_h^1 \phi(t) \in \Vh$, $0 \le t \le T$, as the $H^1$-elliptic projection, and $\hat \mu_h(t) \in \Vh$ by solving the elliptic variational problems
\begin{align} \label{eq:hatmuh}
\la\hat \mu_h(t) - \mu(t),w_h\ra - \gamma \la \nabla \hat \phi_h(t) - \nabla \phi(t), \nabla w_h\ra - \la f'(\hat \phi_h(t)) - f'(\phi(t)), w_h\ra = 0
\end{align}
for all $w_h \in \Vh$ and $0 \le t \le T$. Since this problem is linear in $\hat \mu_h(t)$ and finite-dimensional, existence of a unique solution follows immediately, e.g. by the Lax-Milgram lemma. 
For this choice of approximations $(\hat \phi_h, \hat \mu_h)$, we have the following error estimates. 
\begin{lemma} \label{lem:projerr}
Let (A1)--(A3) hold, $(\phi, \mu)$ be a periodic weak solution of \eqref{eq:ch1}--\eqref{eq:ch2} with regular initial value $\phi(0) \in H^3_p(\Omega)$, and let  $\hat \phi_h$, $\hat \mu_h$ be defined as above. 
Then 
\begin{align*}
\|\phi(t) - \hat \phi_h(t)\|_1 &\leq Ch^2 \|\phi(t)\|_3, \\
\|\dt\phi(t) - \dt\hat \phi_h(t)\|_{-1,h} &\leq Ch^2\|\dt\phi(t)\|_1, \\
\|\mu(t) - \hat \mu_h(t)\|_1 &\leq C' h^2(\|\mu(t)\|_3 + \|\phi(t)\|_3),
\end{align*}
for a.a. $0 \le t \le T$ with constants $C=C(\Omega,\sigma)$ and $C'=C'(\Omega,\sigma,\gamma,f_2^{(2)}, f_3^{(2)},C_{T}(\|\phi_0\|_3))$.
\end{lemma}

\begin{proof}
The estimates for $\phi-\hat\phi_h$ and $\dt\phi-\dt\hat\phi_h$ follow directly from \eqref{eq:h1porjest}.
We then use  the triangle inequality to split the error in the chemical potential into
\begin{align*}
\|\hat \mu_h - \mu\|_1 \le \|\hat \mu_h - \pi_h^0 \mu\|_1 + \|\pi_h^0 \mu - \mu\|_1.
\end{align*}
With the help of \eqref{eq:l2projest}, the last term can be estimated by $\|\pi_h^0 \mu - \mu\|_1 \le C h^2 \|\mu\|_3$. 
Using the first of the inverse inequalities \eqref{eq:inverse}, the discrete error component can be bounded by
\begin{align*}
\|\hat \mu_h - \pi_h^0 \mu\|_1 
\le C_\sigma h^{-1} \|\hat \mu_h - \pi_h^0 \mu\|_0 ,
\end{align*}
and for the  error in the $L^2$-norm, we can deduce from \eqref{eq:l2proj} that 
\begin{align*}
\|\hat \mu_h - \pi_h^0  \mu\|_0^2 
&= (\hat \mu_h - \pi_h^0  \mu, \hat \mu_h - \pi_h^0 \mu) 
=  (\hat \mu_h - \mu, \hat \mu_h - \pi_h^0  \mu), 
\end{align*}
since $w_h = \hat \mu_h - \pi_h^0 \mu \in \Vh$. 
We can then use \eqref{eq:hatmuh} with this test function $w_h$, to see that 
\begin{align*}
(\hat \mu_h - \hat \mu, w_h)
&= \gamma (\nabla (\hat \phi_h - \phi), \nabla w_h) + (f'(\hat \phi_h)-f'(\phi),w_h) \\
&= \gamma (\phi - \hat \phi_h, w_h) + (f'(\hat \phi_h)-f'(\phi),w_h),
\end{align*}
where we used the particular choice of $\hat \phi_h = \pi_h^1 \phi$ and 
\eqref{eq:hatphih}, to replace the gradient term in the second step. 
Proceeding with standard arguments, we then obtain
\begin{align*}
(\hat \mu_h - \hat \mu, w_h)
&\le \gamma \|\phi - \hat \phi_h\|_0 \| w_h\|_0 + \|f'(\hat \phi_h)-f'(\phi)\|_0 \|w_h\|_0 \\
&\le C(f_2^{(2)}, f_3^{(2)},C_{T},\Omega)\|\hat \phi_h - \phi\|_0 \|w_h\|_0.
\end{align*}
To estimate the nonlinear term, we here used the mean value theorem and the polynomial bounds for $f''$ as well as $\|\phi\|_{0,\infty} + \|\hat \phi_h\|_{0,\infty} \le C \|\phi\|_2$.
In summary, we thus obtain 
\begin{align*}
\|\hat \mu_h - \mu\|_1 \le C h^2 (\| \mu\|_3 + \|\phi\|_3),
\end{align*}
with constant $C$ independent of the mesh size and uniform for all $0 \le t \le T$. 
\end{proof}

\subsection{Error estimates}

Using that $(\phi,\mu)$ solves \eqref{eq:weak1}--\eqref{eq:weak2} and the definition of $(\hat \phi_h,\hat \mu_h)$, one can see that  \eqref{eq:weak1hp}--\eqref{eq:weak2hp} is satisfied with residuals $\hat r_{2,h}=0$ and
\begin{align} \label{eq:r1h}
\langle \hat r_{1,h},v_h\rangle 
&= \la\dt \hat \phi_h - \dt \phi, v_h\ra
  + \la b(\phi_h) \nabla (\hat \mu_h -\mu), \nabla v_h\ra 
  + \la (b(\phi_h) - b(\phi)) \nabla \mu, \nabla v_h\ra.
\end{align}
By the properties of the discrete dual norm $\norm{\cdot}_{-1,h}$ and standard approximation error estimates, see Lemma \ref{lem:projerr}, the residual $\hat r_{1,h}$ can further be bounded by 
\begin{align*}
\|r_{1,h}\|_{-1,h}^2
&\le C\|\dt \hat \phi_h - \dt \phi\|_{-1,h}^2
+ C(b_2) \|\nabla \hat \mu_h - \nabla \mu\|_0^2 \\
&\qquad 
+ C(b_3)  \|\nabla \mu\|_{0,3}^2 (\|\hat \phi_h - \phi\|_{0,6}^2 + \|\phi_h - \hat \phi_h\|_{0,6}^2 ) \\
&\le C h^4 (\|\dt \phi\|_{1}^2 + \|\mu\|_{3}^2 + (1+ \|\mu\|_{1,3}^2)\|\phi\|_3^2 ) +  C'\|\mu\|_{1,3}^2 \, \E_\alpha(\phi_h|\hat \phi_h),
\end{align*}
with appropriate constants $C$, $C'$ depending only on bounds on the coefficients, the domain $\Omega$, the mesh regularity, and the constant $C_{T}(\|\phi_0\|_3)$ for the solution in Lemma \ref{lem:weak}.
We can now utilize Lemma~\ref{lem:semi} to obtain the following bounds for the discrete error.
\begin{lemma}
Let (A1)--(A4) hold and $(\phi,\mu)$ be a regular periodic weak solution with initial value $\phi_0\in H_p^3(\Omega)$. Furthermore, let $(\hat \phi_h, \hat \mu_h)$ be the discrete approximations from above and let $(\phi_h,\mu_h)$ be the solution of Problem~\ref{prob:semi} with initial value $\phi_{0,h}=\pi_h^1 \phi_0$. 
Then 
\begin{align*}
\|\phi_h - \hat \phi_h&\|_{L^\infty(H^1_p)}^2 + \|\nabla \mu_h - \nabla \hat \mu_h\|^2_{L^2(L^2)} \le C'_{T}(\|\phi_0\|_3) h^4,
\end{align*}
with constant $C'_T(\|\phi_0\|_3)$ independent of the meshsize $h$.
\end{lemma}
\begin{proof}
From the discrete stability estimate of Lemma~\ref{lem:semi} and the bounds for the residual derived above, we may deduce that 
\begin{align*}
&\E_\alpha(\phi_h(t)|\hat \phi_h(t)) + \int_0^t \D_{\phi_h(s)}(\mu_h(s)|\hat\mu_h(s)) ds 
 \le C  \E_\alpha(\phi_h(0)|\hat \phi_h(0)) + C' \int_0^t \|\hat r_{1,h}(s)\|^2_{-1,h} ds \\
 & \le C'' h^4 \int_0^t \norm*{\dt\phi}_1^2 + \norm*{\mu}_3^2 + (1+ \norm{\mu}^2_{1,3})\norm*{\phi}_3^2 ds + C'''\int_0^t \|\mu\|_{1,3}^2 \E_\alpha(\phi_h(s) | \hat \phi_h(s)) ds.
\end{align*}
Since we assumed $\mu \in L^2(W^{1,3})$, the last term can be eliminated via the Gronwall inequality \eqref{eq:gronwall}, which we here employ with the choices
$u(t) = \E_\alpha(\phi_h(t)|\hat \phi_h(t))$, 
$\beta(t) = C''' \|\mu(t)\|^2_{1,3}$,
and
$\alpha(t)=C \E_\alpha(\phi_h(0)|\hat \phi_h(0)) - \int_0^t \D_{\phi_h(s)}(\mu_h(s)|\hat \mu_h(s)) \, ds + C'' h^4 \int_0^t \norm*{\dt\phi}_1^2 + \norm*{\mu}_3^2 + (1+ \norm{\mu}^2_{1,3})\norm*{\phi}_3^2 ds$. 
The assertion then follows by using the lower bounds \eqref{eq:lower_bound_rel} and \eqref{eq:lower_bound_rel_dis} for the relative energy and dissipation functionals.
\end{proof}

By combination of the previous estimates we now immediately obtain the following error bounds for the Galerkin semi-discretization with quadratic finite elements. 
\begin{theorem} \label{thm:main2}
Let (A1)--(A4) hold and let $(\phi,\mu)$ denote the unique periodic weak solution of \eqref{eq:ch1}--\eqref{eq:ch2} with regular initial value $\phi(0)=\phi_0 \in H^3_p(\Omega)$. Moreover, let $(\phi_h,\mu_h)$ be the corresponding semi-discrete solution of Problem~\ref{prob:semi} with initial value $\phi_{h,0} = \pi_h^1 \phi_0$. Then 
\begin{align*}
&\| \phi - \phi_h\|^2_{L^\infty(H^1_p)} + \| \mu - \mu_h\|^2_{L^2(H^1_p)}  \le C'_{T}(\|\phi_0\|_3) h^4    
\end{align*}
with a constant $C'_T(\|\phi_0\|_3)$ independent of the meshsize $h$. 
\end{theorem}

\begin{remark}
Let us note that the convergence rates in the theorem are optimal with respect to the approximation properties of quadratic finite elements. Moreover, the regularity assumption on the initial value is already necessary for the predicted convergence rates. The convergence result therefore is \emph{order optimal} and \emph{sharp}, i.e., obtained under minimal smoothness assumptions on the problem data.
\end{remark}

\section{Fully discrete approximation} \label{sec:5}

We now turn to the time discretization, for which we again employ a variational method. 
For a given step size $\tau=T/N$, $N \in \NN$, we define discrete time points $t^n:=n\tau$ and denote by $\Itau:=\{0=t^0,t^1,\ldots,t^N=T\}$ the corresponding partition of the time interval $[0,T]$. 
We write $\Pi_N(\Itau;\Vh)$ for the space of piecewise polynomials of degree $k$ over the time grid $\Itau$ with values in $\Vh$, and denote by $\Pi_N^c(\Itau;\Vh) = \Pi_N(\Itau;\Vh) \cap C(0,T;\Vh)$ the corresponding sub-space of continuous functions. 
Furthermore, we use a bar symbol $\bar g$ to denote piecewise constant functions of time. 

We are going to search for approximations $\phi_{h,\tau}$, $\bmu_{h,\tau}$ for $(\phi,\mu)$ in the spaces 
\begin{align*}
\WW_{h,\tau}(0,T):= \Pi_1^c(\Itau;\Vh) 
\quad \text{and} \quad 
\QQ_{h,\tau}(0,T):=\Pi_0(\Itau;\Vh).
\end{align*}
Let us emphasize that functions in $\WW_{h,\tau}(0,T)$ are continuous in time and piecewise linear, while functions $\bar q_{h,\tau} \in \QQ_{h,\tau}$ are piecewise constant in time, which is designated by the bar symbol. 
The fully discrete approximation for \eqref{eq:ch1}--\eqref{eq:ch2} then reads as follows.
\begin{problem} \label{prob:full}
Let $\phi_{0,h}\in \Vh$ be given. Find $\phi_{h,\tau}\in \WW_{h,\tau}(0,T)$, $\bmu_{h,\tau} \in \QQ_{h,\tau}(0,T)$ such that $\phi_{h,\tau}(0)=\phi_{0,h}$ and for all test functions $\bar v_{h,\tau}, \bar w_{h,\tau} \in \QQ_{h,\tau}$ and $n \ge 1$, there holds
\begin{align} 
 \int_{t^{n-1}}^{t^{n}} \la\dt \phi_{h,\tau}, \bar v_{h,\tau}\ra + \la b(\phi_{h,\tau}) \nabla \bmu_{h,\tau}, \nabla  \bar v_{h,\tau}\ra ds &= 0, \label{eq:pg1}
\\
 \int_{t^{n-1}}^{t^{n}} \la \bmu_{h,\tau}, \bar w_{h,\tau}\ra - \la\gamma \nabla \phi_{h,\tau}, \nabla  \bar w_{h,\tau}\ra - \la f'(\phi_{h,\tau}),\bar w_{h,\tau}\ra ds &= 0. \label{eq:pg2}
\end{align}
\end{problem}

\medskip 

\begin{remark}
By the discontinuity of the test functions $\bar v_{h,\tau}$, $\bar w_{h,\tau}$ in time, the fully discrete method amounts to an implicit time-stepping scheme, similar to the Crank-Nicolson or average vector field methods; see \cite{Akrivis11} and \cite{TierraGuillenGonzales15} for details. 
\end{remark}

Before we proceed, let us briefly discuss the well-posedness of the fully discrete scheme.
\begin{lemma}
Let (A1)--(A4) hold. 
Then for any $\phi_{0,h} \in \Vh$ and any $\tau>0$, Problem \ref{prob:full} has at least one solution.
Moreover, any solution $(\phi_{h,\tau},\bmu_{h,\tau})$ of \eqref{eq:pg1}--\eqref{eq:pg2} satisfies identities $\int_\Omega \phi_{h,\tau}(t^n) dx = \int_\Omega \phi_{0,h} dx$ and $\E(\phi_{h,\tau}(t^n)) + \int_0^{t^n} \D_{\phi_{h,\tau}}(\bmu_{h,\tau}) ds = \E(\phi_{0,h}) $ for all $0 \le t^n \le T$,
and as a direct consequence, we obtain uniform bounds 
\begin{align}
\|\phi_{h,\tau}\|_{L^\infty(H^1)} + \|\bmu_{h,\tau}\|_{L^2(H^1)} \le C(\|\phi_{0,h}\|_1).
\end{align}
\end{lemma}
\begin{proof}
Conservation of mass and dissipation of energy follow again by testing the variational identities \eqref{eq:pg1}--\eqref{eq:pg2}, now with $(\bar v_{h,\tau}, \bar w_{h,\tau})=(1,0)$ and $(\bar v_{h,\tau}, \bar w_{h,\tau})=(\bar\mu_{h,\tau},  \dt \phi_{h,\tau})$, which are admissible test functions 
in \eqref{eq:pg1}--\eqref{eq:pg2}.
To show existence, we use an induction argument. Let $\phi_{h,\tau}(t^{n-1})$ be given. 
Then in the $n$th time step, only the function values $\phi_{h}^n:=\phi_{h,\tau}(t^{n})$ and $\mu_h^{n-1/2}:=\bmu_{h,\tau}(t^n-\tau/2) \in \Vh$ need to be determined. 
From the discrete energy-dissipation identity, the bounds for the coefficients, and the equivalence of norms on finite dimensional spaces, one can deduce that potential solutions are necessarily bounded. Existence of a solution for the $n$th time step then follows from Brouwer's fixed-point theorem. 
The uniform bounds for the solution, finally, follow directly from the energy-dissipation identity and using \eqref{eq:lower_bound_rel} and \eqref{eq:lower_bound_rel_dis}.
\end{proof}

\begin{remark}
The uniqueness of the discrete solution can be shown under an appropriate restriction $\tau \le \tau_0(h)$ on the time step size. 
In Section~\ref{sec:unique_full} below, we will show that uniqueness holds for $\tau \le c h^{\alpha}$ with some $\alpha \le 1$, if the solution $(\phi,\mu)$ is sufficiently regular. The choice $\tau = c h$, which seems reasonable in view of the convergence rate estimates of Theorem~\ref{thm:fulldisk}, therefore will lead to unique solutions for the fully discrete problem. 
\end{remark}

In the following, we first establish a discrete analogue of the stability estimate derive in Theorem~\ref{thm:main}, and then derive convergence rates for the fully-discrete scheme.

\subsection{Discrete stability estimate}

For any pair $(\hat \phi_{h,\tau},\hbmu_{h,\tau}) \in \WW_{h,\tau}(0,T)\times \QQ_{h,\tau}(0,T)$, we define discrete residuals $(\bar r_{1,h,\tau},\bar r_{2,h,\tau})\in \QQ_{h,\tau}(0,T)\times \QQ_{h,\tau}(0,T)$ via
\begin{align} 
 \int_{t^{n-1}}^{t^{n}} \la\dt \hat\phi_{h,\tau}, \bar v_{h,\tau}\ra + \la b(\phi_{h,\tau}) \nabla \hbmu_{h,\tau}, \nabla  \bar v_{h,\tau}\ra ds &=: \la \bar r_{1,h,\tau},\bar v_{h,\tau}\ra, \label{eq:discpert1}
\\
 \int_{t^{n-1}}^{t^{n}} \la\hbmu_{h,\tau}, \bar w_{h,\tau}\ra - \la\gamma \nabla \hat\phi_{h,\tau}, \nabla  \bar w_{h,\tau}\ra - \la f'(\hat\phi_{h,\tau}),\bar w_{h,\tau}\ra ds &=: \la \bar r_{2,h,\tau},\bar w_{h,\tau}\ra, \label{eq:discpert2}
\end{align}
for all test functions $\bar v_{h,\tau}, \bar w_{h,\tau} \in \Pi_0(t^{n-1},t^{n};\Vh)$, and all $0 < t^n \le T$. 
Note that the residuals $\bar r_{1,h,\tau}$, $\bar r_{2,h,\tau}$ are defined as piecewise constant functions of time, which we again designate by bar symbols. 
With very similar arguments as used for the derivation of the stability estimates in the previous sections, we now obtain the following result.
\begin{lemma} \label{lem:fullstab}
Let (A1)--(A4) hold and  $(\phi_{h,\tau},\bmu_{h,\tau})$ be a solution of Problem~\ref{prob:full} with stepsize $0<\tau \le \tau_0$ sufficiently small.
Furthermore, let $(\hat \phi_{h,\tau},\hbmu_{h,\tau}) \in \WW_{h,\tau}(0,T) \times \QQ_{h,\tau}(0,T)$ be given and $(\bar r_{1,h,\tau},\bar r_{2,h,\tau})$ denote the corresponding residuals defined by \eqref{eq:discpert1}--\eqref{eq:discpert2}. 
Then
\begin{align*}
 \E_\alpha(\phi_{h,\tau}&({t^{n}})|\hat \phi_{h,\tau}({t^{n}})) + \int_0^{t^{n}} \D_{\phi_{h,\tau}(s)}(\bmu_{h,\tau}(s)|\hbmu_{h,\tau}(s))ds \\
 &\leq e^{c t^n}\E_\alpha(\phi_{h,\tau}(0)|\hat \phi_{h,\tau}(0)) + Ce^{c t^n} \int_0^{t^{n}} \|\bar r_{1,h,\tau}(s)\|_{-1}^2 + \|\bar r_{2,h,\tau}(s)\|_1^2   ds
\end{align*}
for all $0 \le t^n \le T$ with constants $c = c_0 + c_1 \|\dt \hat \phi_{h,\tau}\|_{L^\infty(L^2)}$, and $c_0,c_1,C$ depending only on the bounds for the coefficients, the domain $\Omega$, and the uniform bounds for $(\phi_{h,\tau}, \mu_{h,\tau})$ and $(\hat \phi_{h,\tau},\hat \mu_{h,\tau})$ in $L^\infty(H^1) \times L^2(H^1)$. 
\end{lemma}
\begin{remark}
It will become clear from the proof that the energy estimate of Lemma~\ref{lem:fullstab} holds uniformly for all $h>0$ and $0 < \tau \le \tau_0$ with $\tau_0$ only depending on the bounds for the coefficients, the domain $\Omega$, the time horizon $T$, as well as the uniform bounds for $(\phi_{h,\tau}, \mu_{h,\tau})$ and $(\hat \phi_{h,\tau},\hat \mu_{h,\tau})$ in $L^\infty(H^1) \times L^2(H^1)$ and on the bound for $\|\dt \hat \phi_{h,\tau}\|_{L^\infty(L^2)}$. 
\end{remark}
\begin{proof}
By the fundamental theorem of calculus, we obtain
\begin{align*}
    &\E_\alpha(\phi_{h,\tau}|\hat \phi_{h,\tau})\big|_{t^{n-1}}^{t^{n}} 
    = \int_{t^{n-1}}^{t^{n}} \ddt \E_\alpha(\phi_{h,\tau}|\hat \phi_{h,\tau}) \, ds \\
    &= \int_{t^{n-1}}^{t^{n}} \gamma \la \nabla \phi_{h,\tau} - \nabla \hat \phi_{h,\tau}, \nabla \dt \phi_{h,\tau} - \nabla \dt \hat \phi_{h,\tau}\ra 
    +\la f'(\phi_{h,\tau}) - f'(\hat\phi_{h,\tau}), \dt\phi_{h,\tau} - \dt\hat\phi_{h,\tau} \ra \\
    &\qquad \qquad  + \alpha \la \phi_{h,\tau} - \hat \phi_{h,\tau}, \dt \phi_{h,\tau} - \dt \hat \phi_{h,\tau}\ra \\
    &\qquad \qquad 
   + \la f'(\phi_{h,\tau}) - f'(\hat \phi_{h,\tau}) - f''(\hat \phi_{h,\tau}) (\phi_{h,\tau} - \hat \phi_{h,\tau}), \dt \hat \phi_{h,\tau}\ra \, ds \\
   &= \int_{t^{n-1}}^{t^{n}} \la \bmu_{h,\tau} - \hbmu_{h,\tau} + \bar r_{2,h,\tau}, \dt \phi_{h,\tau} - \dt \hat \phi_{h,\tau}\ra 
   + \alpha \la \phi_{h,\tau}- \hat \phi_{h,\tau}, \dt \phi_{h,\tau} - \dt \hat \phi_{h,\tau}\ra \\ 
   &\qquad \qquad 
   + \la f'(\phi_{h,\tau}) - f'(\hat \phi_{h,\tau}) - f''(\hat \phi_{h,\tau}) (\phi_{h,\tau} - \hat \phi_{h,\tau}), \dt \hat \phi_{h,\tau}\ra \, ds = (*).
\end{align*}
In the last step, we utilized the identites \eqref{eq:pg2} and \eqref{eq:discpert2} with the admissible test function $\bar w_{h,\tau} = \dt \phi_{h,\tau} \in \QQ_{h,\tau}$.
Since $\dt \phi_{h,\tau} - \dt \hat \phi_{h,\tau}$ is piecewise constant in time, we can replace 
$$
\int_{t^{n-1}}^{t^n} \alpha \la \phi_{h,\tau}- \hat \phi_{h,\tau}, \dt \phi_{h,\tau} - \dt \hat \phi_{h,\tau}\ra dt 
= 
\int_{t^{n-1}}^{t^n} \alpha \la \bar\pi_\tau^0 \phi_{h,\tau}- \bar\pi_\tau^0 \hat \phi_{h,\tau}, \dt\phi_{h,\tau} - \dt \hat \phi_{h,\tau}\ra dt
$$
in the previous expression, where $\bar\pi_\tau^0 : \WW_{h,\tau} \to \QQ_{h,\tau}$ denotes the $L^2$-orthogonal projection in time onto piecewise constants.  
Employing $\bar v_{h,\tau} = \mu_{h,\tau} - \hat \mu_{h,\tau} + \bar r_{2,h,\tau} + \alpha \bar\pi_\tau^0 (\phi_{h,\tau} - \hat \phi_{h,\tau}) \in \QQ_{h,\tau}$ as a test function in the identities \eqref{eq:pg1} and \eqref{eq:discpert1}, we further obtain 
\begin{align*}
(*)
  &=  \int_{t^{n-1}}^{t^{n}} -\la b(\phi_{h,\tau}) \nabla (\bmu_{h,\tau} - \hbmu_{h,\tau}), \nabla(\bmu_{h,\tau}-\hbmu_{h,\tau} + \bar r_{2,h,\tau})\ra - \la \bar r_{1,h,\tau}, \bmu_{h,\tau} - \hbmu_{h,\tau} + \bar r_{2,h,\tau} \ra \\
& \qquad \qquad - \alpha \la b(\phi_{h,\tau}) \nabla (\bmu_{h,\tau} - \hbmu_{h,\tau}), \nabla (\phi_{h,\tau} - \hat \phi_{h,\tau}) \ra - \alpha \la \bar r_{1,h,\tau}, \phi_{h,\tau} - \hat \phi_{h,\tau} \ra \\
 & \qquad \qquad + \la f'(\phi_{h,\tau}) - f'(\hat \phi_{h,\tau}) - f''(\hat \phi_{h,\tau})(\phi_{h,\tau} - \hat \phi_{h,\tau}), \dt \hat \phi_{h,\tau} \ra \, ds.
\end{align*}
At this point, we can start to estimate the individual terms in the same manner, as in the proof of Theorem \ref{thm:main}.
In this way, we arrive at
\begin{align*}
 \E_\alpha(\phi_{h,\tau}|\hat \phi_{h,\tau})&\big|_{t^{n-1}}^{t^{n}} + \int_{t^{n-1}}^{t^{n}} \D_{\phi_{h,\tau}(s)}(\bmu_{h,\tau}(s)|\hbmu_{h,\tau}(s))ds \\
 &\leq \int_{t^{n-1}}^{t^{n}} c \E_\alpha(\phi_{h,\tau}(s)|\hat \phi_{h,\tau}(s))ds+ \int_{t^{n-1}}^{t^{n}} C_1\|\bar r_{1,h,\tau}(s)\|_{-1}^2 + C_2\|\bar r_{2,h,\tau}(s)\|_1^2 ds,
\end{align*}
where we may choose $c=c_1 + c_2  \|\dt \hat \phi_{h,\tau}(s)\|_{L^\infty(L^2)}$ here with $c_1,c_2$ and also $C_1,C_2$ depending only on the uniform $L^\infty(H^1) \times L^2(H^1)$ bounds for $(\phi_{h,\tau},\bmu_{h,\tau})$ and $(\hat \phi_{h,\tau},\hbmu_{h,\tau})$, as well as on the bounds for the coefficients, and the domain $\Omega$, and the time horizon. 
Using the uniform $L^\infty(H^1)$ bounds for $\phi_{h,\tau}$ and $\hat \phi_{h,\tau}$ and \eqref{eq:lower_bound_rel}, we can further estimate the first term on the right hand side of the previous inequality by
\begin{align*}
    \int_{t^{n-1}}^{t^n} \E_\alpha(\phi_{h,\tau}|\hat \phi_{h,\tau}) ds \le c(\gamma) \tau (\E_\alpha(\phi_{h,\tau}(t^n)) + \E_{\alpha}(\phi_{h,\tau})(t^{n-1})). 
\end{align*}
Under the assumption that $\tau \le 1/(2 c(\gamma) c) =: \tau_0$, we can rewrite the estimate into 
\begin{align*}
u^n + b^n \le e^{\lambda} u^{n-1} + d^n,    
\end{align*}
with $u^n=\E_\alpha(\phi_{h,\tau}(t^n)|\hat \phi_{h,\tau}(t^n))$, $b^n =  e^{\gamma \tau} \int_{t^{n-1}}^{t^n} \D_{\phi_{h,\tau}}(s) ds$, $d^n = e^{\gamma \tau} \int_{t^{n-1}}^{t^{n}} C_1\|\bar r_{1,h,\tau}(s)\|_{-1}^2 + C_2\|\bar r_{2,h,\tau}(s)\|_1^2 ds$, and $e^{\gamma \tau} = \frac{1+c(\gamma) c \tau}{1-c(\gamma) c \tau}$, which corresponds to $\gamma \approx 2 c(\gamma) c$. 
The assertion then follows by the discrete Gronwall-inequality \eqref{eq:discgronwall} and the bounds \eqref{eq:lower_bound_rel} and \eqref{eq:lower_bound_rel_dis} for the relative energy and dissipation functionals. 
\end{proof}

\subsection{Auxiliary results}
Similar to the semi-discrete case, we will utilize certain projections to define suitable approximations $\hat \phi_{h,\tau}$ and $\hbmu_{h,\tau}$ for solutions $(\phi,\mu)$ to \eqref{eq:ch1}--\eqref{eq:ch2} that allow us to take advantage of the discrete stability estimate. 
To this end let 
$$
\I_\tau^1:H^1(0,T)\to \Pi_1^c(\Itau), \qquad  \I_\tau^1 u(t^n)=u(t^n)
$$
denote the piecewise linear interpolation with respect to time. Furthermore, let 
$$
\bar \pi_\tau^0 : L^2(0,T) \to \Pi_0(\I_\tau), \qquad \bar \pi_\tau^0 u(t) = \frac{1}{\tau} \int_{t^{n-1}}^{t^n} u(t) dt, \qquad t \in (t^{n-1}, t^n),
$$
be the $L^2$-orthogonal projection to piecewise constant functions in time.
For later reference, we summarize some important properties of these operators. 
\begin{lemma}
For $u \in W^{r,q}(0,T)$, $0 \le r \le 1$, $1 \le p \le q \le \infty$, there holds 
\begin{align} \label{eq:timprojest}
    \|u - \bar\pi_\tau^0 u\|_{L^p(0,T)} \le C \tau^{1/p-1/q+r} \|u\|_{W^{r,q}(0,T)},
\end{align}
and for $u \in W^{r,q}(0,T)$ with $1 \le r \le 2$ and $1 \le p \le q  \le \infty$, one has
\begin{align}\label{eq:timinterpest}
    \|u - \I_\tau^1 u\|_{L^p(0,T)} &\leq C\tau^{1/p-1/q+r}\norm{u}_{W^{r,q}(0,T)}.
\end{align}
Moreover, the interpolation and projection operators commute with differentiation, i.e., 
\begin{align} \label{eq:commuting} 
\dt (\I_\tau^1 u) = \bar \pi_\tau^0 (\dt u).
\end{align}
\end{lemma}
\begin{proof}
The proof for these standard results can be found, e.g., in \cite{BrennerScott}. 
\end{proof}
The interpolation operator naturally extends to vector valued functions and we use the same symbol in that case. 
For the piecewise-constant $L^2$-projection we can show the following estimate for the product error; see appendix \ref{sec_proj} for a proof.
\begin{lemma} \label{lem:l2timeproduct}
Let $u,v\in W^{2,p}(0,T)$ and $\bar a = \bar \pi^0_\tau a$ denotes the $L^2$-orthogonal projection onto piecewise constants. 
Then
\begin{align}  \label{eq:l2timeproduct}
\|\bar u\bar v -\overline{uv}\|_{L^p(0,T)} 
\le C \tau^2 \|u\|_{W^{2,p}(0,T)} \|v\|_{W^{2,p}(0,T)} 
\end{align}
with a constant $C$ independent of $\tau$ and $p$ as well as the functions $u$ and $v$.
\end{lemma}
\noindent 
As fully discrete approximations $(\hat \phi_{h,\tau},\hbmu_{h,\tau}) \in \WW_{h,\tau} \times \QQ_{h,\tau}$ for solutions $(\phi,\mu)$ of \eqref{eq:ch1}--\eqref{eq:ch2}, to be used in the subsequent error analysis, we now define
\begin{align} \label{eq:fullproj}
\hat \phi_{h,\tau} = \I_\tau^1 \pi_h^1 \phi 
\qquad \text{and} \qquad 
\hbmu_{h,\tau} = \bar \pi_\tau^0 \pi_h^0 \mu.
\end{align}
For this particular choice, we can make the following observation.
\begin{lemma}
Let $(\phi,\mu)$ denote a sufficiently regular periodic weak solution of \eqref{eq:weak1}--\eqref{eq:weak2}, and let $(\hat \phi_{h,\tau},\hbmu_{h,\tau})$ be defined as above. Then \eqref{eq:pg1}--\eqref{eq:pg2} holds with residuals
\begin{align*}
\int_{t^{n-1}}^{t^n} \la \bar r_{1,h,\tau}, \bar v_{h,\tau}\ra ds 
    &= \int_{t^{n-1}}^{t^n} \la \dt (\pi_h^1 \phi - \phi), \bar v_{h,\tau}\ra + \la b(\phi_{h,\tau}) \nabla \hbmu_{h,\tau} - b(\phi) \nabla \mu, \nabla \bar v_{h,\tau}\ra ds \\
\int_{t^{n-1}}^{t^n} \la  \bar r_{2,h,\tau}, \bar w_{h,\tau} \ra ds 
    &= \int_{t^{n-1}}^{t^n} \la \hbmu_{h,\tau} - I_\tau^1 \mu, \bar w_{h,\tau}\ra 
       + \gamma \la \nabla (\hat \phi_{h,\tau} - \I_\tau^1 \phi), \nabla \bar w_{h,\tau}\ra \\
    & \qquad \qquad \qquad \qquad \qquad \quad  + \la f'(\hat \phi_{h,\tau}) - \I_\tau^1 f'(\phi), \bar w_{h,\tau}\ra \, ds.  
    \end{align*}
\end{lemma}
\begin{proof}
Testing \eqref{eq:weak1} with $v=\bar v_{h,\tau}$ and integration over time yields 
\begin{align*}
    \int_{t^{n-1}}^{t^n} \la \dt \phi, \bar v_{h,\tau}\ra + \la b(\phi) \nabla \mu, \nabla \bar v_{h,\tau}\ra &= 0.
\end{align*}
The first identity then follows by taking the difference of this equation with \eqref{eq:pg1}, and noting that
\begin{align*}
    \int_0^{t^n} \la \dt \hat \phi_{h,\tau}, \bar v_{h,\tau}\ra \, ds 
    &= \int_0^{t^n} \la \dt \I_\tau^1 \pi_h^1 \phi, \bar v_{h,\tau}\ra \, ds 
    = \int_0^{t^n} \la \dt \pi_h^1 \phi, \bar v_{h,\tau}\ra \, ds, 
\end{align*}
which follows from \eqref{eq:commuting} and the fact that $\bar v_{h,\tau}$ is piecewise constant in time. 
Testing equation \eqref{eq:weak2} at time $t^{n-1}$ and $t^n$ with $\bar w_{h,\tau}$, and noting that $\int_{t^{n-1}}^{t^n} a(t) \bar b(t) dt = \frac{\tau}{2} (a_1(t^n) + a_2(t^n)) b(t^{n}-\tau/2)$ for all $a \in \Pi_1(t^{n-1},t^n)$ and $\bar b_0 \in \Pi_0(t^{n-1},t^n)$, one can see that 
\begin{align*}
\int_{t^{n-1}}^{t^n} \la \I_\tau^1 \mu, \bar w_{h,\tau} \ra + \gamma \la \nabla \I_\tau^1 \phi, \nabla \bar w_{h,\tau}\ra + \la \I_\tau^1 f'(\phi), \bar w_{h,\tau}\ra \, ds = 0.
\end{align*}
Combination with with \eqref{eq:pg2} then yields the second identity.
\end{proof}

As a next step, we derive bounds for the discrete residuals in terms of interpolation and projection errors. 
For ease of notation, we will write $W^{k,p}(X) = W^{k,p}(a,b;X)$ for different choices of the time interval $(a,b)$, which will be clear from the context. 
\begin{lemma} \label{lem:pqresidualest}
Let $(\phi,\mu)$ denote a sufficiently regular weak solution of \eqref{eq:weak1}--\eqref{eq:weak2}. 
Then 
\begin{align*}
\int_{t^{n-1}}^{t^n} \|\bar r_{1,h,\tau}\|_{-1,h}^2 ds &\le C_0(\phi,\mu)\tau^4 + C_1(\phi,\mu)h^4  +  C(b,\mu)\int_{t^{n-1}}^{t^n}  \E_\alpha(\phi_{h,\tau}(s)|\hat \phi_{h,\tau}(s)) ds, \\
\int_{t^{n-1}}^{t^n} \|\bar r_{2,h,\tau}\|_1^2 ds &\leq C_2(\phi,\mu) \tau^4 + C_3(\phi) h^4
\end{align*}
for all $0 < t^n \le T$ with constants $C(\cdot)$ independent of $h$, $\tau$, and $t^n$.
All Bochner-norms $W^{k,p}(X)=W^{k,p}(t^{n-1},t^n;X)$ here refer to the time interval $(t^{n-1},t^n)$ under consideration.
\end{lemma}
\begin{proof}

Since $\bar v_{h,\tau}$ is piecewise constant in time, we can use \eqref{eq:commuting}, the definition of $\hat \phi_{h,\tau}$, and the bounds for the $H^1$-projection error, to estimate the residual by
\begin{align*}
    \int_{t^{n-1}}^{t^n} \|\bar r_{1,h,\tau}\|^2_{-1,h} ds &\le C\int_{t^{n-1}}^{t^n}  \|\dt (\pi_h^1 \phi - \phi)\|^2_{-1} +  \|\overline{b(\phi_{h,\tau})} \nabla \hbmu_{h,\tau} - \overline{b(\phi) \nabla \mu}\|^2_0 \, ds \\
    &\le C h^4 \|\dt \phi\|_{L^2(H^1_p)}^2 + C(*)^2.
\end{align*}
Here and in the following, we use $\overline{a} = \bar \pi_\tau^0 a$ to abbreviate the projection onto piecewise constant functions in time. 
The remaining term can be further estimated by
\begin{align*}
    (*)^2 & \le \|\overline{b(\phi_{h,\tau})} \nabla (\hbmu_{h,\tau} - \overline{\mu})\|_{L^2(L^2)}^2 + \|(\overline{b(\phi_{h,\tau})} - \overline{b(\phi)}) \nabla \overline{\mu}\|_{L^2(L^2)}^2 \\
    & \qquad \qquad \qquad  \qquad + \|\overline{b(\phi) \nabla \mu} - \overline{b(\phi)}\;\overline{\na\mu}\|_{L^2(L^2)}^2 
    = (i)+(ii)+(iii).
\end{align*}
Using the boundedness of $b$, the definition of $\hbmu_{h,\tau}$, and the stability and error estimates for the $L^2$-projection \eqref{eq:l2projest}, we immediately obtain
\begin{align*}
    (i) \le C(b_2) \|\nabla (\pi_h^0 \mu - \mu)\|_{L^2(L^2)}^2 \le C'(b_2) h^4 \|\mu\|_{L^2(H^3_p)}^2. 
\end{align*}
For the second term, we use a triangle inequality,  the error bounds \eqref{eq:h1porjest} for the $H^1$-projection, the interpolation error estimate \eqref{eq:timinterpest}, and the lower bound \eqref{eq:lower_bound_rel} for the relative energy. In summary, this leads to 
\begin{align*}
    (ii) &\le C(b_3) \|\phi_{h,\tau} - \phi\|_{L^2(L^6)}^2 \|\mu\|_{L^\infty(W^{1,3}_p)}^2 \\
    &\le C(b_3) \|\mu\|_{L^\infty(W^{1,3}_p)} (h^4 \|\phi\|_{L^2(H^3_p)}^2 + \tau^4 \|\phi\|_{H^2(H^1_p)}^2 + \gamma \E_\alpha(\phi_{h,\tau}|\hat \phi_{h,\tau})),
\end{align*}
For the third term, we observe that this is a second order approximation on the midpoint of the time interval and using the estimate \eqref{eq:l2timeproduct} we obtain
\begin{align*}
    (iii) \le C\tau^4\norm{b(\phi)\na\mu}_{H^2(L^2)}^2 \leq C(b_2,b_3,\norm{\mu}_{L^\infty(W^{1,3}_p)}) \tau^4 (\|\mu\|_{H^2(H^1_p)}^2 + \norm{\phi}_{H^2(H^1)}^2).
\end{align*}
By combination of the previous estimates, we thus obtain 
\begin{align*}
    &\int_{t^{n-1}}^{t^n} \|\bar r_{1,h,\tau}\|^2_{-1,h} ds 
    \le C_0(\phi,\mu)\tau^4 + C_1(\phi,\mu)h^4 + C(b,\mu) \int_{t^{n-1}}^{t^n}  \E_\alpha(\phi_{h,\tau}(s)|\hat \phi_{h,\tau}(s)) ds,
\end{align*}
with $C_0=C(\|\phi\|_{H^2(H^1_p)},\|\mu\|_{H^2(H^1_p)})$, $C_1(\phi,\mu)=C(\|\dt \phi\|_{L^2(H^1_p)},\|\mu\|_{L^2(H^3_p)}, \|\phi\|^2_{L^2(H^3_p)})$ and constant $C(b,\mu)=C(b_3,\|\mu\|_{L^\infty(W^{1,3}_p)})$ independent of $h$ and $\tau$.

Before turning to the bound for the second residual, let us observe that
\begin{align} \label{eq:aux}
\int_{t^{n-1}}^{t^n}     \la \nabla (\hat \phi_{h,\tau} - \I_\tau^1 \phi), \nabla \bar w_{h,\tau} \ra \, ds
&= \int_0^{t^n} \la \I_\tau^1 \phi - \hat \phi_{h,\tau}, \bar w_{h,\tau}\ra \, ds,
\end{align}
which follows from the definition of $\hat \phi_{h,\tau}$ and the variational characterization \eqref{eq:hatphih} of $\pi_h^1$.
The second residual can then be expressed equivalently in strong form as 
\begin{align*}
    \bar r_{2,h,\tau} = (\overline{\pi_h^0 \mu} - \overline{\I_\tau^1 \pi_h^0 \mu}) + (\overline{\I_\tau^1 \phi} - \overline{\hat \phi_{h,\tau}}) + (\overline{f'(\hat \phi_{h,\tau})} - \overline{\I_\tau^1 f'(\phi)}),
\end{align*}
where $\overline{g} = \bar \pi_\tau^0 g$ denotes the piecewise constant projection of $g$ with respect to time. 
This pointwise representation allows us to estimate
\begin{align*}
\int_{t^{n-1}}^{t^n} \|\bar r_{2,h,\tau}\|_1^2 ds &\leq \|\pi_h^0 \mu - \I_\tau^1 \pi_h^0 \mu\|^2_{L^2(H^1_p)} + \|\I_\tau^1 \phi - \hat \phi_{h,\tau}\|^2_{L^2(H^1_p)} \\
& \qquad \qquad \qquad + \|f'(\hat \phi_{h,\tau}) - \I_\tau^1 f'(\phi)\|^2_{L^2(H^1_p)}  
= (i) + (ii) + (iii) .   
\end{align*}
We again estimate the individual terms separately. 
For the first, we use the contraction property of the $L^2$-projection $\pi_h^0$ and the interpolation error estimate \eqref{eq:timinterpest} to obtain 
\begin{align*}
    (i) \le \|\mu - \I_\tau^1 \mu\|_{L^2(H^1_p)}^2 \le C \tau^4 \|\mu\|_{H^2(H^1_p)}^2.
\end{align*}
For the second term, we employ the error estimate \eqref{eq:h1porjest} for the $H^1$-projection $\pi_h^1$ to get 
\begin{align*}
    (ii) \le \|\phi - \pi_h^1 \phi\|_{L^\infty(H^1_p)}^2 \le C h^4 \|\phi\|_{L^\infty(H^3_p)}^2.
\end{align*}
For the third term, we use the fact that 
$\phi$ and its discrete counter part $\hat\phi_{h,\tau} = \I_\tau^1 \phi_h^1 \phi$ can be uniformly bounded in $L^\infty(0,T;W^{1,\infty}_p(\Omega))$.
Therefore, all terms $f^{(k)}(\cdot)$ appearing in the following can be bounded uniformly by a constant $C(f)$.
This leads to 
\begin{align*}
    (iii) &\le \|f'(\hat \phi_{h,\tau}) - f'(\phi)\|_{L^2(H^1_p)}^2 + \|f'(\phi) - \I_\tau^1 f'(\phi)\|_{L^2(H^1_p)}^2 \\
    &\le C_1(f) \|\hat\phi_{h,\tau} - \phi\|_{L^2(H^1_p)}^2 + \tau^4\|f'(\phi)\|_{H^2(H^1_p)}^2 \\ 
    &\le C(f) (h^4 \|\phi\|_{L^2(H^3_p)} + \tau^4 \|\phi\|_{H^2(H^1_p)}) + \tau^4\|f'(\phi)\|_{H^2(H^1_p)}^2.
\end{align*}
A quick inspection of the last term shows that its evaluation involves up to cubic products of $\phi$ and its derivatives, with the highest order terms given by
\def\dtt{\partial_{tt}}
$\phi^2\dtt \nabla \phi$, $(\dt \phi)^2 \nabla \phi$, and $\phi\dt \phi \nabla \dt \phi$, respectively.
This allows to establish the following bounds 
\begin{align*}
\|f'(\phi)\|_{H^2(H^1_p)} 
& \leq C(f) (1+\|\phi\|_{H^2(H^1)} +  \|\phi\|_{H^1(H^3)})^3.
\end{align*}
In summary, the second residual can thus be bounded by
\begin{align*}
\int_{t^{n-1}}^{t^n} \|\bar r_{2,h,\tau}\|_1^2 ds 
\le C_1(\phi,\mu) \tau^4 + C_2(\phi) h^4, 
\end{align*}
with solution dependent constants $C_2(\phi,\mu) = C(\|\mu\|_{H^2(H^1_p)},\|\phi\|_{H^2(H^1_p)},  \|\phi\|_{H^1(H^3_p)})$ and $C_3(\phi) = C(\|\phi\|_{L^\infty(H^3_p)}, \|\phi\|_{L^2(H^3_p)})$ independent of $h$ and $\tau$.
\end{proof}

\subsection{Error estimates}

Together with the discrete stability estimate of Lemma~\ref{lem:fullstab} and a Gronwall-type argument, similar as already used in the proof of that result, 
we can now obtain the following convergence rate estimates. 
\begin{theorem}
\label{thm:fulldisk}
Let $(\phi,\mu)$ 
denote a regular periodic weak solution of \eqref{eq:ch1}--\eqref{eq:ch2} with initial value $\phi_0\in H^3_p(\Omega)$ satisfying additionally
\begin{align*}
    \phi&\in H^{2}(0,T;H^1_p(\Omega))\cap H^1(0,T;H^3_p(\Omega)), \\ 
    \mu& \in H^2(0,T;H^1_p(\Omega))\cap  L^\infty(0,T;W^{1,3}_p(\Omega)),
\end{align*}
and let $(\phi_{h,\tau},\bmu_{h,\tau})$ be a solution of \eqref{eq:pg1}--\eqref{eq:pg2} with initial value $\phi_{h,\tau}(0) = \pi_h^1 \phi_0$. 
Then 
\begin{align*}
    \max_{t^n\in\mathcal{I}_\tau}\norm{\phi_{h,\tau}(t^n) - \phi(t^n)}_{1}^2 + \norm*{\bmu_{h,\tau} - \bar\mu}_{L^2(0,T;H^1_p)}^2 \leq C'_{T}(h^4 + \tau^4),
\end{align*}
with $C'_T$ depending on the norms of the solution $(\phi,\mu)$, but independent of $h$ and $\tau$.
\end{theorem}
\begin{proof}
We may proceed almost verbatim to the proof of Lemma~\ref{lem:fullstab} and insert the above estimates for the residual terms, to see that
\begin{align*}
 \E_\alpha(\phi_{h,\tau}|\hat \phi_{h,\tau})&\big|_{t^{n-1}}^{t^{n}} + \int_{t^{n-1}}^{t^{n}} \D_{\phi_{h,\tau}(s)}(\bmu_{h,\tau}(s)|\hbmu_{h,\tau}(s))ds \\
 &\leq C_1' \tau^4 + C_2' h^4 + \int_{t^{n-1}}^{t^{n}} C(b,\mu) \E_\alpha(\phi_{h,\tau}(s)|\hat \phi_{h,\tau}(s)) ds, 
\end{align*}
with $C(b,\mu)= c_0 + c_1 \|\dt \hat \phi_{h,\tau}\|_{L^\infty(L^2)} + c_2 \|\mu\|_{L^\infty(W^{1,3}_p)}$ bounded uniformly in time.
The proof of the assertion then follows in the same manner as that of Lemma~\ref{lem:fullstab}. 
Note that it suffices to consider the case that $\tau \le \tau_0$ is sufficiently small, since for large $\tau$ the result already follows from the a-priori estimates . 
\end{proof}

\subsection{Uniqueness of the fully discrete solution}
\label{sec:unique_full}

Using the previous estimates, we now show that uniqueness of the fully discrete solution can be obtained under a mild restriction on the time step size.
We start with the observation that under the conditions of the previous theorem, $\bmu_{h,\tau}$ is uniformly bounded in $L^\infty(W^{1,3}_p)$. To see this, note that
\begin{align*}
\|\nabla \bmu_{h,\tau}\|_{L^\infty(W^{1,3}_p)} 
&\le \|\bmu_{h,\tau} - \pi_h^0 \bar \mu\|_{L^\infty(W^{1,3}_p)} + \|\pi_h^0 \bar \mu - \bar \mu\|_{L^\infty(W^{1,3}_p)} + \|\bar \mu\|_{L^\infty(W^{1,3}_p)}.
\end{align*}
The last two terms are uniformly bounded by assumption and standard projection error estimates. For the first term on the right hand side, we use the second of the inverse inequalities \eqref{eq:inverse} with $p=3,q=2,d\le3$ in space and with $p=\infty,q=2,d=1$ in time, 
as well as the convergence estimates of the previous theorem, to see that 
\begin{align*}
\|\bmu_{h,\tau} - \pi_h^0 \bar \mu\|_{L^\infty(W^{1,3}_p)} 
&\le C \tau^{-1/2} h^{-1/2}  \|\bmu_{h,\tau} - \pi_h^0 \bar \mu\|_{L^2(H^1_p)} 
\le C_1' h^{-1/2}\tau^{3/2} +  C_2' h^{3/2} \tau^{-1/2}.
\end{align*}
For any choice $c h^{3} \le \tau \le C h^{1/3}$, one can thus conclude that $\|\bmu_{h,\tau}\|_{L^\infty(W^{1,3}_p)} \le C'$.

Uniqueness of the discrete solution can now be deduced as follows:
Let the assumptions of Theorem~\ref{thm:fulldisk} be valid. Furthermore, let $(\phi_{h,\tau},\bmu_{h,\tau})\in \WW_{h,\tau}(0,T)\times \QQ_{h,\tau}(0,T)$ and $(\hat\phi_{h,\tau},\hbmu_{h,\tau})\in \WW_{h,\tau}(0,T)\times \QQ_{h,\tau}(0,T)$ denote two solutions of Problem \ref{prob:full} with the same initial data $\phi_{h,\tau}(0)=\hat\phi_{h,\tau}(0)$ and with time step size $c h^{3} \le \tau \le C h^{1/3}$ and $\tau \le \tau_0$ sufficiently small. 
Then the residuals defined by \eqref{eq:discpert1}--\eqref{eq:discpert2} are $\bar r_{2,h,\tau}=0$ and
\begin{align*}
    \int_{t^{n-1}}^{t^{n}} \la \bar r_{1,h,\tau},\bar v_{h,\tau} \ra ds =\int_{t^{n-1}}^{t^{n}}\la (b(\phi_{h,\tau})-b(\hat\phi_{h,\tau}))\nabla\hbmu_{h,\tau},\bar v_{h,\tau} \ra ds, 
\end{align*}
for all $\bar v_{h,\tau}\in \Pi_0(t^{n-1},t^{n};\Vh)$. 
Using the bounds for the coefficients and \eqref{eq:lower_bound_rel}, 
the residual term $\bar r_{1,h,\tau}$ can be further estimated by
\begin{align*}
   \int_{t^{n-1}}^{t^{n}} \|\bar r_{1,h,\tau}\|_{-1,h}^2 ds
   \leq C(b_3) \|\hbmu_{h,\tau}\|^2_{L^\infty(W^{1,3}_p)} \int_{t^{n-1}}^{t^n} \E(\phi_{h,\tau}(s)|\hat\phi_{h,\tau}(s))ds.
\end{align*}
The last term can now be handled by a Gronwall-type argument, similar as in the proof of Lemma~\ref{lem:fullstab} and Theorem~\ref{thm:fulldisk}. Together with $\phi_{h,\tau}(0)=\hat \phi_{h,\tau}(0)$, we thus obtain
\begin{align*}
 \E_\alpha(\phi_{h,\tau}(t^n)|\hat \phi_{h,\tau}(t^n)) + \int_0^{t^n} \D_{\phi_{h,\tau}}(\bmu_{h,\tau}|\hbmu_{h,\tau}) ds
 & \le 0.
 \end{align*}
By the lower bounds \eqref{eq:lower_bound_rel} and \eqref{eq:lower_bound_rel_dis} for the relative energy and dissipation terms, this implies that $\nabla \bmu_{h,\tau} \equiv \nabla \hbmu_{h,\tau}$ and $\phi_{h,\tau}(t^n) = \hat\phi_{h,\tau}(t^n)$ for all $n$, from which one can deduce that $\phi_{h,\tau} \equiv \hat \phi_{h,\tau}$ and $\bmu_{h,\tau} \equiv \hbmu_{h,\tau}$.

\begin{remark}
A brief inspection of the arguments reveal, that the regularity assumptions on the true solution could be somewhat relaxed, which will however lead to tighter bounds $c h^{\alpha} \le \tau \le C h^{1/\alpha}$ with $1 \le \alpha \le 3$ for the admissible time step sizes. 
The choice $\tau = c h$, seems reasonable and leads to a uniqueness result under minimal regularity assumptions.
If the mobility function $b(\phi) \equiv b$ is independent of the concentration, then the above considerations become obsolete, since the relevant terms in the stability estimate vanish. 
\end{remark}

\section{Numerical validation} \label{sec:6}

For illustration of our theoretical results, in particular, of the convergence rate estimates of Theorem~\ref{thm:main2} and \ref{thm:fulldisk}, we report in this section about some numerical results for a typical test problem, which is specified as follows: 
We choose a polynomial potential
\begin{align*}
    f(\phi) = 0.3(\phi-0.99)^2(\phi+0.99)^2,
\end{align*}
define the mobility function
\begin{align*}
    b(\phi) = (1-\phi)^2(1+\phi)^2 + 10^{-3},
\end{align*}
and choose $\gamma = 0.003$ as the interface parameter.
All assumptions (A1)--(A3) of Section~2 are thus satisfied. 
As computational domain, we use the unit cell $\Omega = (0,1)^2$, and the system \eqref{eq:ch1}--\eqref{eq:ch2} is complemented by periodic boundary conditions.
%
%
We finally choose
\begin{align*}
    \phi_0(x,y) = 0.2\sin\left( 4\pi x\right)\sin\left( 2\pi y \right) +0.2
\end{align*}
as initial value for the phase fraction. 
%

For all our computations, we use the fully discrete approximation of Problem~\ref{prob:full} on a sequence of uniformly refined triangulations $\Th$ and equidistant grids $\mathcal{I}_\tau $ in time. 
%
%
Some snapshot of the computed phase fraction $\phi_{h,\tau}$ are depicted in Figure \ref{fig:evo}.
\begin{figure}[ht!]
\centering
\bigskip 
\bigskip 
\footnotesize
\begin{tabular}{ccc}
t=0 & t=0.1 & t=0.2 \\
\includegraphics[trim={2.3cm 1.4cm 1.5cm 1.0cm},clip,scale=0.25]{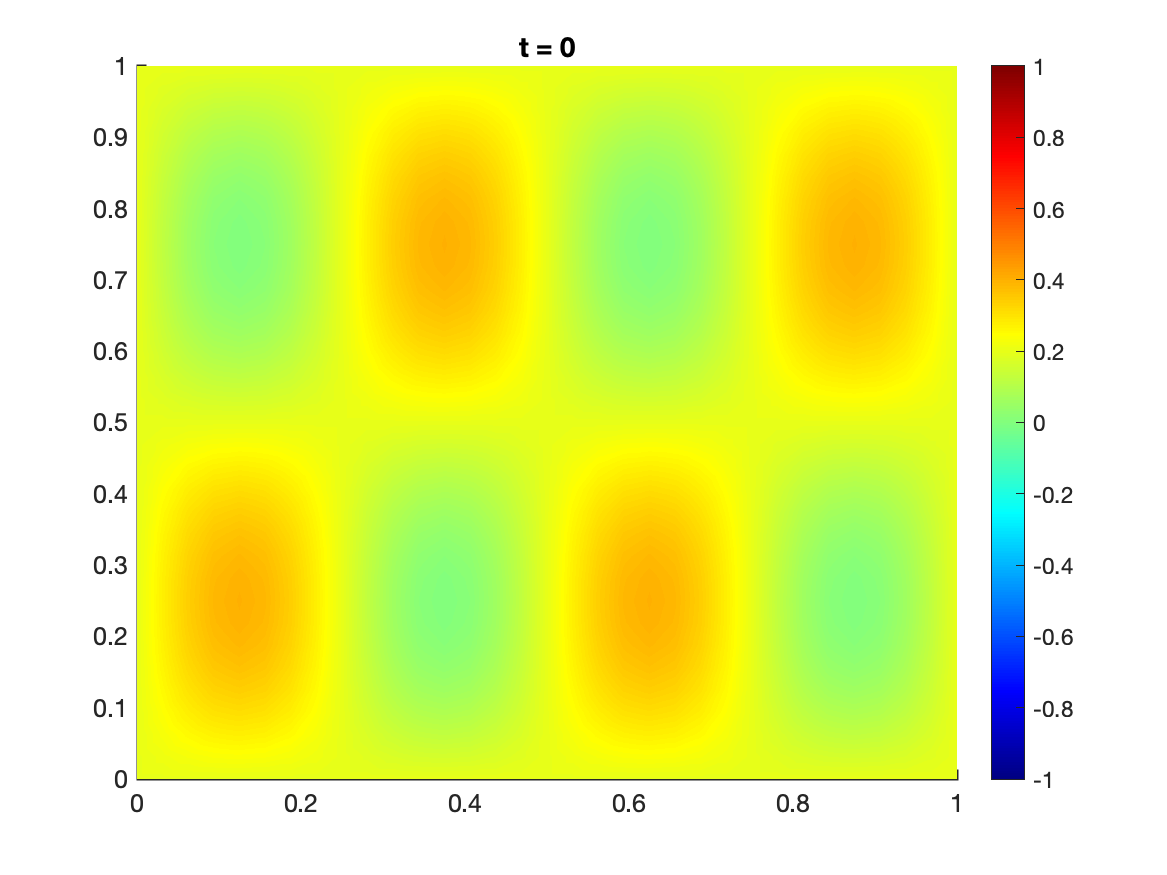} 
&
\includegraphics[trim={2.3cm 1.4cm 1.5cm 1.0cm},clip,scale=0.25]{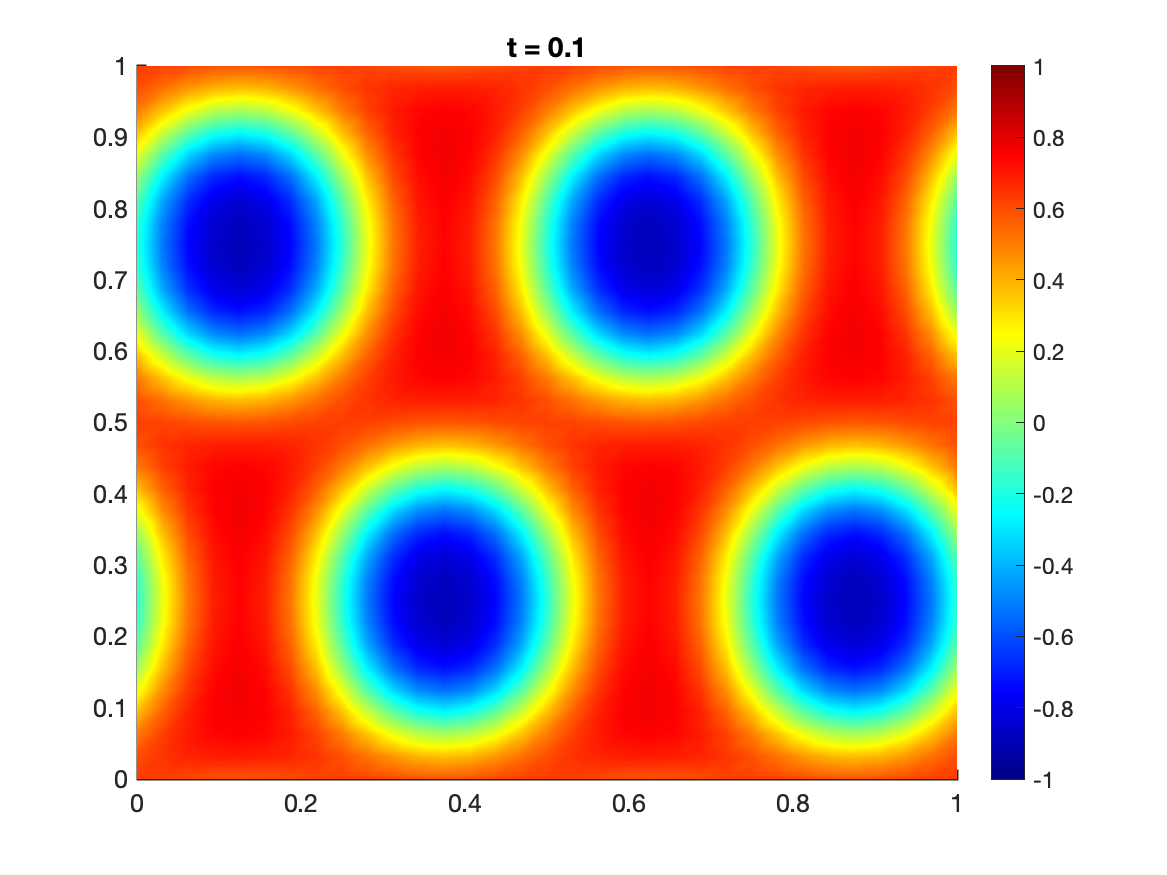}  
&
\includegraphics[trim={2.3cm 1.4cm 1.5cm 1.0cm},clip,scale=0.25]{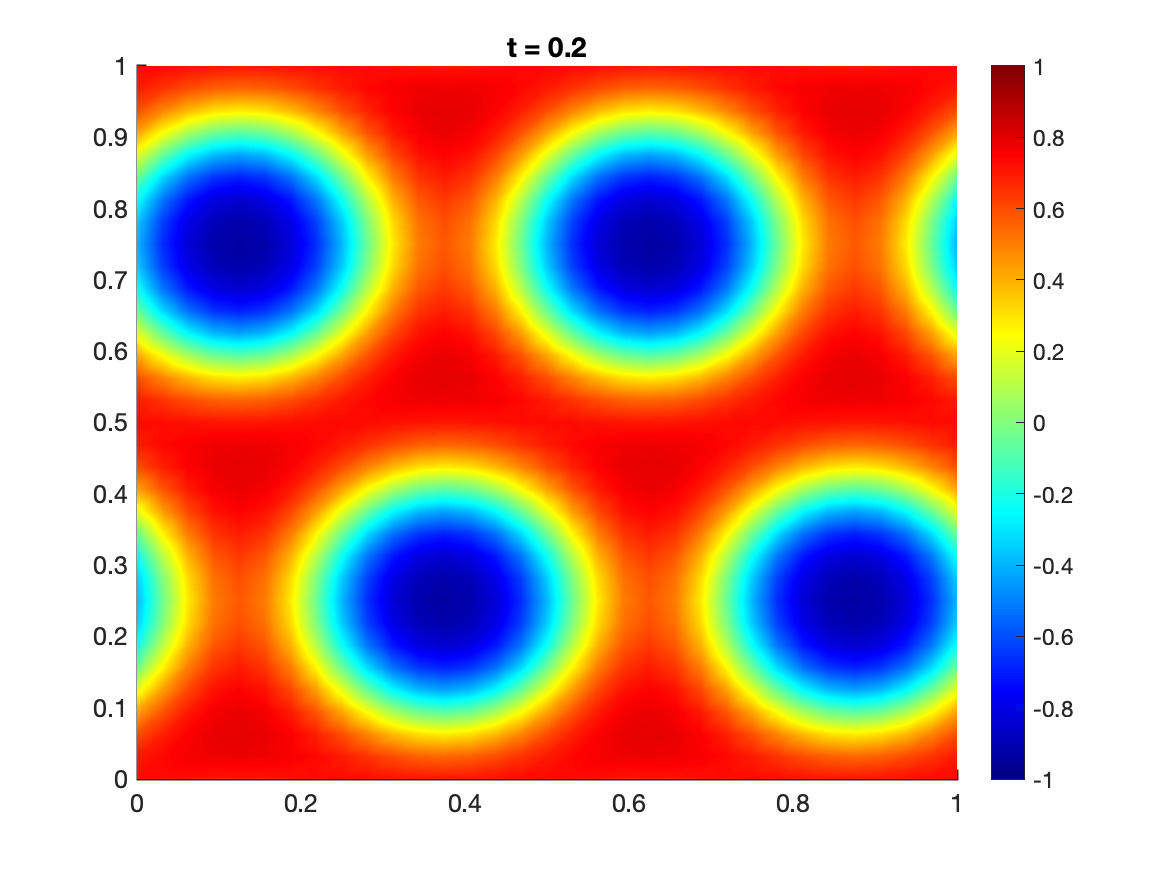} 
\\[1em]
\includegraphics[trim={2.3cm 1.4cm 1.5cm 1.0cm},clip,scale=0.25]{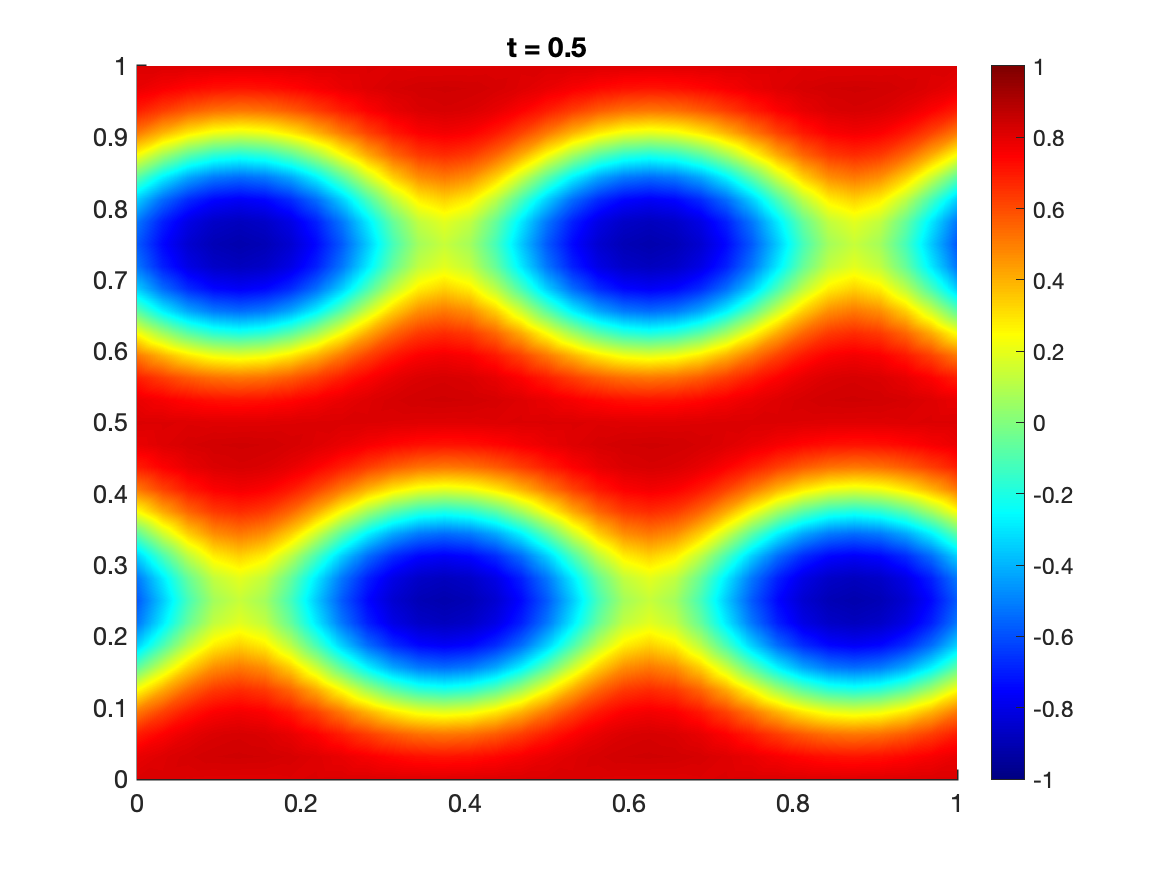}  
&
\includegraphics[trim={2.3cm 1.4cm 1.5cm 1.0cm},clip,scale=0.25]{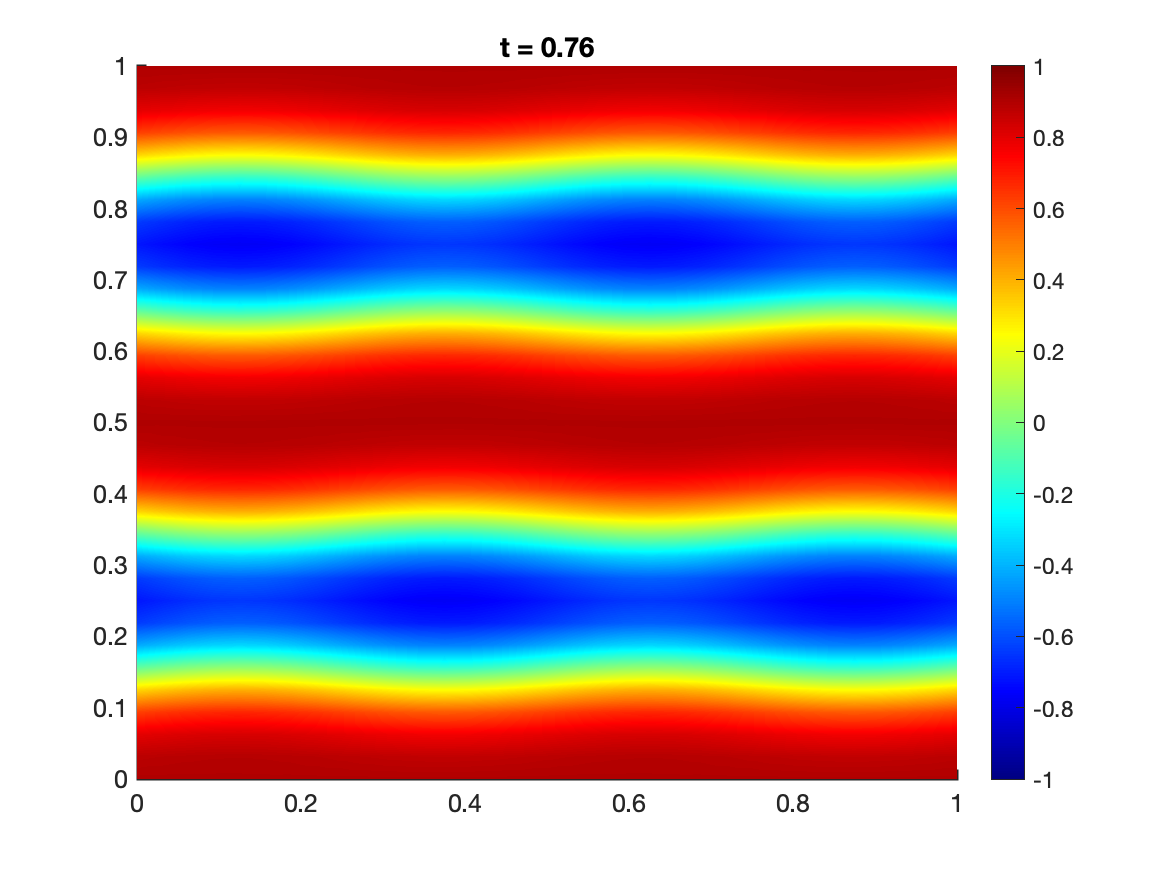} 
&
\includegraphics[trim={17.4cm 1.4cm 1.5cm 1.4cm},clip,scale=0.125]{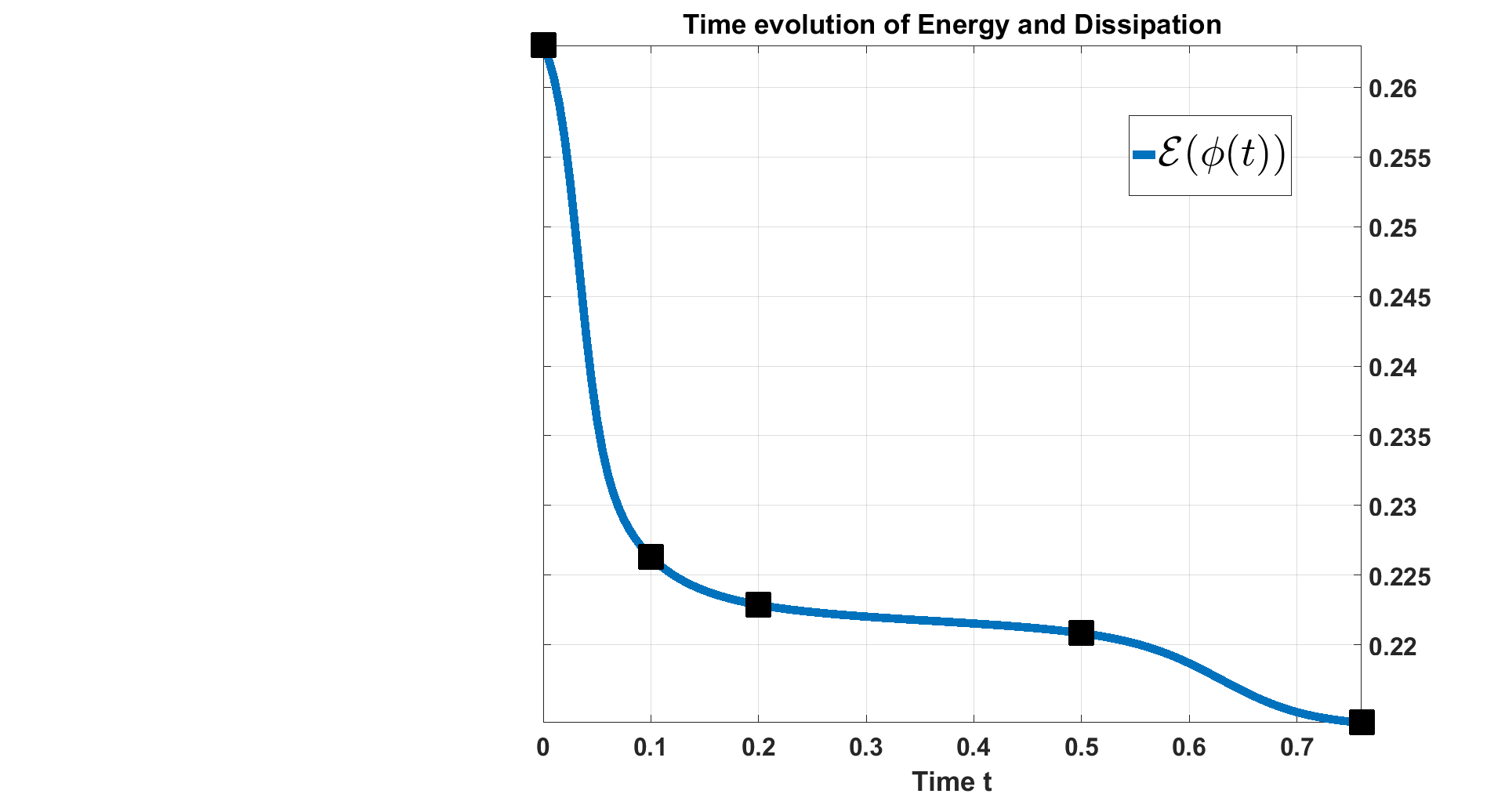} \\
t=0.5 & t=0.76 & energy \\
\end{tabular}
\caption{Snapshots of the phase fraction $\phi_{h,\tau}$ and evolution of the energy.\label{fig:evo}
}
\end{figure} 
One can clearly observe the expected evolution from a rather uniform distribution to an almost completely separated configuration.
As predicted by our theoretical results, the solution remains smooth over the whole time interval used for our simulations. 

We now turn to the convergence rates. Since no analytical solution is available, the discretization error is estimated by comparing the computed solutions $(\phi_{h,\tau},\bar\mu_{h,\tau})$ with those  computed on uniformly refined grids. 
The error quantities for the fully-discrete scheme are thus defined by 
\begin{align*}
e_{h,\tau} &= \max_{t^n\in\mathcal{I}_{\tau}}\norm{\phi_{h,\tau}(t^n) - \phi_{h/2,\tau/2}(t^n)}_{H^1_p}+ \norm*{\bmu_{h,\tau} - \bmu_{h/2,\tau/2}}_{L^2(0,T;H^1_p)}.
\end{align*}
In order to evaluate the convergence rates of the semi-discretization, we choose a very small step size $\tau^*$, and refer to $\phi_h := \phi_{h,\tau^*}$ as the semi-discrete approximation in the following. The corresponding error quantities are then defined as 
\begin{align*}
e_{h} &= \max_{t^n\in\mathcal{I}_{\tau}}\norm{\phi_{h,\tau^*}(t^n) - \phi_{h/2,\tau^*}(t^n)}_{H^1_p}+ \norm*{\bmu_{h,\tau^*} - \bmu_{h/2,\tau^*}}_{L^2(0,T;H^1_p)}.
\end{align*}
In Table~\ref{tab:rates}, we report the results of our computations obtained on a sequence of uniformly refined meshes with mesh size $h_k=2^{-(3+k)}$, $k=0,\ldots,3$ and time steps $\tau_k = 0.16 \cdot h_k$. For the results concerning the semi-discretization, the time step is chosen $\tau^*=0.16 \cdot 2^{-9}$.
Since nested grids are used in all our computations, the error quantities defined above can be computed exactly. 
\begin{table}[ht!]
\centering
\small
\begin{tabular}{c||c|c||c|c}
$ k $ & $ e_h $  & eoc & $e_{h,\tau}$ & eoc \\
\hline
$ 0 $   & $1.4794 \cdot 10^{-0}$  &   ---   & $1.5183\cdot 10^{-0}$ &  ---      \\
$ 1 $   & $3.7373 \cdot 10^{-1}$  & 1.98  & $3.7896\cdot 10^{-1}$ & 2.00     \\
$ 2 $   & $9.2554 \cdot 10^{-2}$  & 2.01  & $9.2797\cdot 10^{-2}$ & 2.02     \\
$ 3 $   & $2.3622 \cdot 10^{-2}$  & 1.97  & $2.3795\cdot 10^{-2}$ & 1.96     \\
$ 4 $   & $5.9391 \cdot 10^{-3}$  & 1.99  & $6.0902\cdot 10^{-3}$ & 1.96      
\end{tabular}
\bigskip
\caption{Errors and convergence rates for the computational results obtained with the semi-discrete and fully-discrete approximations. \label{tab:rates}} 
\end{table}
As usual, the experimental order of convergence (eoc) is computed by comparing comparing the errors of two consecutive refinements.
In perfect agreement with the theoretical predictions of Theorem~\ref{thm:main2} and \ref{thm:fulldisk}, we observe second order convergence for the errors. The proposed method thus is of second order in space and time.

\section{Discussion} \label{sec:7}

In this paper, we studied the stability, regularity, and uniqueness of solutions to the Cahn-Hilliard equation with concentration-dependent mobility. The variational characterization of weak solutions and relative energy estimates were used as the main ingredients of our analysis, and the latter greatly simplified the handling of nonlinear terms in the problem. 
The basic tools of our analysis are applicable almost verbatim to discretization schemes based on variational principles, i.e., Galerkin finite-element approximations in space and Petrov-Galerkin approximation in time. 
The variational time discretization, which is tightly related to the average vector field methods, leads to fully-implicit schemes which, however, can be solved efficiently by Newton-iterations, and which allows for a structured and transparent error analysis. 
The convergence results obtained in the paper are of optimal order and the result for the semi-discretization is sharp concerning regularity requirements of the solution. Some additional regularity is required for the fully-discrete scheme, which can be explained by the lack of strong stability of the Petrov-Galerkin tim discretization; see \cite{AndreevSchweitzer2014} for details. 
In principle, the proposed schemes can be extended immediately to higher order in space and time. Further investigations in this direction and the extension to more complex multiphase problems, e.g., the Cahn-Hilliard Navier-Stokes equations, will be topics of future research.

{\footnotesize

\section*{Acknowledgement}
Support by the German Science Foundation (DFG) via TRR~146:  \emph{Multiscale Simulation Methods for Soft Matter Systems}, project~C3, and SPP~2256: \emph{Variational Methods for Predicting Complex Phenomena in Engineering Structures and Materials}, project Eg-331/2-1 is gratefully acknowledged. M.L.\ is grateful to the Gutenberg Research College, University Mainz for supporting her research.}

\bigskip

\bibliographystyle{abbrv}
\bibliography{relenergy}

\newpage

\appendix

\section{Gronwall lemmas}

Let us start with recalling the following classical version of Gronwall's lemma.
\begin{lemma} \label{lem:gronwall}
Let $T>0$, $v,g \in C[0,T]$ and $\lambda \in L^1(0,T)$ be given. Further assume that 
\begin{align*}
    v(t) \le g(t) + \int_0^t \lambda(s) v(s) ds, \quad 0 \le t \le T,
\end{align*}
and that $\lambda(t) \ge 0$ for a.a. $0 \le t \le T$. Then
\begin{align} \label{eq:gronwall}
    v(t) \le g(t) + \int_0^t g(s) \lambda(s) e^{\int_s^t \lambda(r) dr} ds, \qquad 0 \le t \le T.
\end{align}
\end{lemma}
A proof can be found in \cite[Ch.~29]{Wloka}. 
A similar result also holds on the discrete level. 
\begin{lemma} \label{lem:discgronwall}
Let $(u_n)_n$, $(b_n)_n$, $(c_n)_n$, and $(\lambda_n)_n$ be given positive sequences, satisfying
\begin{align*}
    u_n + b_n \le e^{\lambda_n} u_{n-1} + c_n, \qquad n \ge 0.
\end{align*}
Then 
\begin{align} \label{eq:discgronwall}
    u_n + \sum_{k=1}^n e^{\sum_{j={k+1}}^{n} \lambda_j} b_k 
    \le e^{\sum_{j=1}^n \lambda_j} u_0 + \sum_{k=0}^n e^{\sum_{j={k+1}}^{n} \lambda_j} c_k, \quad n > 0. 
\end{align}
\end{lemma}
\begin{proof}
The result follows immediately by induction.
\end{proof}

\section{Proof of Lemma~\ref{lem:l2timeproduct}} 
\label{sec_proj}

We start with considering a single element $J=(t^{n-1},t^n)$ and show that 
\begin{align} \label{eq:est}
\|\bar u \bar v - \overline{u v}\|_{0,p} \le C \tau^2 (\|u\|_{2,p} \|v\|_{1,\infty} + \|u\|_{1,\infty} \|v\|_{2,p}),
\end{align}
where $\|\cdot\|_{k,p} = \|\cdot\|_{W^{k,p}(J)}$ and $\bar a =\bar \pi^0_\tau a$ denotes the average of $a$ over $J$. 
In addition, we denote by $\tilde a = a(t^{n-1/2})$ the constant interpolant at $t^{n-1/2} = (t^n + t^{n-1})/2$. 
Then we have 
\begin{align*}
   \|\bar u \bar v - \overline{uv}\|_{0,p} 
   &\le \|\bar u \bar v - \tilde u \bar v\|_{0,p}
      + \|\tilde u \bar v - \tilde u \tilde v\|_{0,p}
      + \|\tilde u \tilde v - \widetilde{ u v}\|_{0,p}
      + \|\widetilde{u v} - \overline{u v}\|_{0,p} \\ 
  &= (i) + (ii) + (iii) + (iv).      
\end{align*}
In order to bound the individual terms, we utilize the super-closeness estimate
\begin{align} \label{eq:super}
    \|\bar a - \tilde a\|_{0,p} \le C \tau^2 \|a\|_{2,p},
\end{align}
which follows by observing that $\bar a - \tilde a = 0$ for $a \in P_1(J)$ and using the Bramble-Hilbert lemma and a scaling argument; see \cite{BrennerScott} for details. 
We can then estimate the first term in the above error expansion by 
\begin{align*}
(i) \le \|\bar u - \tilde u\|_{0,p} \|\bar v\|_{0,\infty} \le C \tau^2 \|u\|_{2,p} \|v\|_{0,\infty},
\end{align*}
and in a similar manner, we see that $(ii) \le C h^{k+2} \|u\|_{0,\infty} \|v\|_{k+2,p}$.
The third term vanishes identically, i.e., $(iii)=0$, and using \eqref{eq:super} again, the last term can be bounded by
\begin{align*}
(iv) \le C \tau^{2} \|uv\|_{2,p} 
\le C' \tau^{2} (\|u\|_{2,p} \|v\|_{1,\infty)}
+ \|u\|_{1,\infty}\|v\|_{2,p} ).
\end{align*}
This proves the estimate \eqref{eq:est} for one single element $J=(t^{n-1},t^n)$.
The global projection estimate \eqref{eq:l2timeproduct} then follows by summation over the elements and using the bounds for the continuous embedding $\|a\|_{L^\infty(0,T)} \le \|a\|_{W^{1,\infty}(0,T)} \le C \|a\|_{W^{2,p}(0,T)}$.
\qed

\section{Regularity}

We now discuss improved regularity results for the weak solution $(\phi,\mu)$ of (\ref{eq:ch1})-(\ref{eq:ch2}) and the initial data are given by $\phi_0\in H^k_p(\Omega)$, $k\in\{2,3\}$.
The basic argument relies on Galerkin approximation and uniform a-priori estimates, which are obtained by testing the discretized variational problems with approximations for higher order derivatives and using energy-type estimates and Gronwall-type inequalities. 

Using simplifications of the results and proofs presented in \cite{Boyer13}, one can see that for initial value $\phi_0\in H^2_p(\Omega)$ the Galerkin approximations $(\phi_N,\mu_N)$ satisfy
\begin{align} \label{eq:boyerreg}
\phi_N &\in L^\infty(0,T;H^2_p(\Omega))\cap L^2(0,T;H^4_p(\Omega)), \\
\dt\phi_N &\in L^2(0,T;L^2(\Omega)), \\
\mu_N &\in L^\infty(0,T;L^2(\Omega))\cap L^2(0,T;H^2_p(\Omega)), 
\end{align}
with uniform bounds for the respective norms, i.e., independent of the level $N$ of the approximation. This immediately leads to the bounds of Lemma \ref{lem:weak} for $k=2$. 
Let us note that in three space dimensions, the maximal time $T$ of validity has to be chosen sufficiently small, depending on the norm of the initial data, while in two space dimensions $T$ can be chosen arbitrary; we refer to \cite{Boyer13} for details.

Now assume that $\phi_0 \in H_p^3(\Omega)$. 
We may then test the Galerkin approximation of the weak formulation \eqref{eq:weak1} with $v_N=-\Delta^3\phi_N$, and obtain
\begin{equation*}
    \ddt \norm{\na\Delta\phi_N}_0^2 - \la \nabla\div(m(\phi_N)\na\mu_N) , \na\Delta^2\phi_N\ra = 0.
\end{equation*}
By elementary computations, one can verify that 
\begin{align*}
&\nabla\div(m(\phi_N)\na\mu_N) \\
&= b(\phi_N)\nabla\Delta\mu_N + 2b'(\phi_N)\nabla\phi_N\Delta\mu_N + b'(\phi_N)\Delta\phi_N\nabla\mu_N + b''(\phi_N)\snorm{\nabla\phi_N}^2\nabla\mu_N    \\
&=(i) + (ii) + (iii) + (iv)    .
\end{align*}
From the regularity result for $k=2$ and standard embedding results, we already know that $\phi_N$ is bounded in $L^\infty(0,T;L^\infty(\Omega))$. 
We can then decompose the first term by
\begin{equation*}
(i)= -b(\phi_N)\nabla\Delta^2\phi_N + \nabla\Delta(f'(\phi_N)), 
\end{equation*}
and further estimate the Laplacian of $f'$ by
\begin{align*}
\norm{\nabla\Delta(f'(\phi_N))}_0^2 &\leq  C(f^{(4)})\norm{\nabla\phi_N}_{0,6}^6 + C(f^{(3)})\norm{\nabla\phi_N\Delta\phi_N}_{0}^2 +  C(f^{(2)})\norm{\nabla\Delta\phi_N}_{0}^2  \\
&\leq  C(f^{(4)})\norm{\nabla\phi_N}_{0,6}^6 + C(f^{(3)})\norm{\nabla\Delta\phi_N}_{0,2}^2\norm{\nabla\phi_N}_{0,4}^2 +  C(f^{(2)})\norm{\nabla\Delta\phi_N}_{0}^2. 
\end{align*}
Using the improved bounds for $\phi_N$ for $k=2$, we can also estimate the other terms by
\begin{align*}
\la (ii), \nabla\Delta^2\phi_N \ra &\leq \delta\norm{\nabla\Delta^2\phi_N}_{0}^2   + C(b_3,\delta)\norm{\Delta\mu_N\nabla\phi_N}_{0}^2 \\
& \leq \delta\norm{\nabla\Delta^2\phi_N}_{0}^2  + C(b_3,\delta)\norm{\Delta\mu_N}_{0,4}\norm{\nabla\phi_N}_{0,4}^2 \\
& \leq 2\delta\norm{\nabla\Delta^2\phi_N}_{0}^2  + C(b_3,\delta)\norm{\Delta\mu_N}_{0,2}^2 \norm{\nabla\phi_N}_{0,4}^2 \\
& \qquad\qquad \qquad \qquad + C(b_3,\delta)\norm{\nabla\Delta(f'(\phi_N))}_{0}\norm{\nabla\phi_N}_{0,4}^2,\\
\la (iii), \nabla\Delta^2\phi_N \ra & \leq \delta\norm{\nabla\Delta^2\phi_N}_{0}^2   + C(b_3,\delta)\norm{\Delta\phi_N\nabla\mu_N}_{0}^2   \\
& \leq \delta\norm{\nabla\Delta^2\phi_N}_{0}^2   + C(b_3,\delta)\norm{\na\Delta\phi_N}_0^2\norm{\nabla\mu_N}_{0,4}^2,\\
\la (iv), \nabla\Delta^2\phi_N \ra & \leq \delta\norm{\nabla\Delta^2\phi_N}_{0}^2 + C(b_4,\delta)\norm{\snorm{\nabla\phi_N}^2\nabla\mu_N}_0^2 \\
& \leq \delta\norm{\nabla\Delta^2\phi_N}_{0}^2 + C(b_4,\delta)\norm{\nabla\phi_N}^4_{0,8}\norm{\nabla\mu_N}_{0,4}^2.
\end{align*}
Setting $y(t)=\norm{\nabla\Delta\phi_N(t)}_0^2$, a combination of the above estimates directly leads to
\begin{align*}
y(t) + (c_0b_1 - 4\delta)\int_0^t \norm{\nabla\Delta^2\phi_N}^2_0 \leq y(0)  + C\int_0^t g(s)y(s) ds + C\int_0^t h(s) ds,
\end{align*}
with $g(s) \leq C + \norm{\na\mu_N}_{0,4}^2 + \norm{\nabla\phi_N}_{0,4}^2$ and $h(s) \leq C + \norm{\na\mu_N}_{0,4}^2\norm{\Delta\phi_N}_2^4 + \norm{\nabla\phi_N}_{0,4}^4$.
From the improved regularity \eqref{eq:boyerreg} for $k=2$, 
one can deduce that $g,h$ are uniformly bounded in $L^1(0,T)$. 
Choosing $\delta$ sufficiently small and applying the Lemma~\ref{lem:gronwall} together with the previous estimates, now leads to
\begin{align}
\phi_N &\in L^\infty(0,T;H^3_p(\Omega))\cap L^2(0,T;H^5_p(\Omega)), \\
\mu_N &\in L^\infty(0,T;H^1_p(\Omega))\cap L^2(0,T;H^3_p(\Omega)),
\end{align}
with uniform bounds (independent of $N$) for the corresponding norms. A straight forward computation further shows that
\begin{align}
    \dt \phi_N \in L^2(0,T;H^1_p(\Omega)),
\end{align}
together with corresponding uniform bounds.
Taking the limit with $N \to \infty$, maybe after choosing a weakly convergent sub-sequence, shows that corresponding bounds also hold for the weak solution $u = \lim_N u_N$. Hence at least one regular weak solution $(\phi,\mu)$ exists satisfying the bounds of Lemma~\ref{lem:weak}. 
Let us emphasize that the bounds hold for all $T>0$ in two space dimensions, while $T>0$ has to be chosen sufficiently small, depending on the problem data, in three dimensions.

\section{Limiting process for the stability estimate}

For ease of notation, we denote the space-time cylinder by $\Omega_T:=\Omega\times(0,T)$ in the following.
Let $(\phi,\mu)\in \WW(0,T)\times\QQ(0,T)$ be given periodic weak solution of the Cahn-Hilliard system \eqref{eq:ch1}--\eqref{eq:ch2}. Then by the mollification procedure proposed by Meyers and Serrin \cite{MeyersSerrin64}, one can construct a sequence 
$(\phi_n,\mu_n)\in \WW(0,T)\cap C^\infty(\Omega_T) \times \QQ(0,T)\cap C^\infty(\Omega_T)$ of smooth approximations, such that
\begin{align*}
\phi_n \rightarrow \phi \text{ in } \WW(0,T)
\qquad \text{and} \qquad 
\mu_n \rightarrow \mu \text{ in } \QQ(0,T)
\qquad \text{with } n \to \infty.    
\end{align*}
Similar as in the proof of Theorem~\ref{thm:main}, we can define residuals $r_{1,n}$, $r_{2,n}$ such that 
\begin{align*}
\la \dt \phi_n(t),v \ra &+ \la b(\phi_n(t))\nabla\mu_n(t),\nabla v \ra =: \la  r_{1,n}(t),v \ra \\
\la \mu_n(t), w \ra &- \gamma\la \nabla\phi_n(t),\nabla w \ra - \la f'(\phi_n(t)),w \ra =: \la  r_{2,n}(t),w \ra
\end{align*}
for all test functions $v,w \in H^1_p(\Omega)$ and all $0 \le t \le T$. 
Since $(\phi,\mu)$ is a periodic weak solution of \eqref{eq:ch1}--\eqref{eq:ch2}, one can immediately see that
\begin{align*}
\la  r_{1,n},v \ra &= \la \dt \phi_n - \dt\phi,v \ra + \la b(\phi_n)\nabla\mu_n - b(\phi)\nabla\mu, \nabla v\ra \\
\la r_{2,n}, w\ra &= \la \mu_n - \mu, w \ra - \la \nabla(\phi_n-\phi),\nabla w \ra - \la f'(\phi_n)-f'(\phi),w \ra
\end{align*}
for all $v,w \in H^1_p(\Omega)$ and a.a. $0 \le t \le T$. From the convergence of $(\phi_n,\mu_n)$ to $(\phi,\mu)$ in the norms stated above, and the assumptions on the coefficients, one can deduce that %
\begin{equation*}
\lim_{n\to\infty}\norm{r_{1,n}}_{-1}^2 = 0, \quad \lim_{n\to\infty}\norm{r_{2,n}}_{1}^2 = 0.   
\end{equation*}
In a similar manner, we choose for given $(\hat\phi,\hat\mu)\in \WW(0,T)\cap W^{1,1}(0,T;L^2(\Omega)\times \QQ(0,T)$ a sequence of smooth approximations $(\hat\phi_m,\hat\mu_m)$ such that
\begin{align*}
 \hat\phi_m \rightarrow \hat\phi \text{ in } \WW(0,T)\cap W^{1,1}(0,T;L^2(\Omega)) \text{ and } \hat\mu_m \rightarrow \hat\mu \text{ in } \QQ(0,T)
\end{align*}
with $n \to \infty$, and define corresponding residuals
\begin{align*}
\la \dt \hat\phi_m,v \ra + \la m(\phi_n)\nabla\hat\mu_m,\nabla v \ra =: \la \tilde r_{1,m},v \ra \\
\la \hat\mu_m, w \ra - \gamma\la \nabla\hat\phi_m,\nabla w \ra - \la f'(\hat\phi_m),w \ra =: \la \tilde r_{2,m},w \ra
\end{align*}
for all $v,w \in H^1_p(\Omega)$ and all $0 \le t \le T$. 
Using \eqref{eq:weak1p}--\eqref{eq:weak2p}, we immediately deduce that 
\begin{align*}
\la  \tilde r_{1,m},v \ra &= \la \hat r_1,v \ra + \la \dt \hat\phi_m - \dt\hat\phi,v \ra + \la b(\phi_n)\nabla(\hat\mu_m -\hat\mu), \nabla v\ra \\
\la \tilde r_{2,m}, w\ra &=  \la \hat r_{2},w \ra +  \la \hat\mu_m - \hat\mu, w \ra - \la \nabla(\hat\phi_m-\hat\phi),\nabla w \ra - \la f'(\hat\phi_m)-f'(\hat\phi),w \ra
\end{align*}
for all $v,w \in H^1_p(\Omega)$ and a.a. $0 \le t \le T$. 
Using the convergence of $(\hat \phi_m,\hat \mu_m)$ towards $(\hat \phi,\hat \mu)$ in the corresponding norms and the assumptions on the parameters, one can see that
\begin{equation*}
\lim_{m\to\infty}\norm{\tilde r_{1,m} - r_{1,m}}_{-1}^2 = 0, \quad \lim_{m\to\infty}\norm{\tilde r_{2,m}- r_{2,m}}_{1}^2 = 0.   
\end{equation*}
With a slight adoption of the proof in Theorem~\ref{thm:main}, we now obtain the stability estimate
\begin{align*}
 \E_\alpha(\phi_n(t)|\hat\phi_m(t)) &+ \int_0^t\D_{\phi_n}(\mu_n(s)|\mu_m(s)) ds \\
 &\leq  C e^{ct}\E_\alpha(\phi_n(0)|\hat\phi_m(0)) +  Ce^{ct}\int_0^t \|\tilde r_{1,m}- r_{1,n}\|^2_{-1} + \|\tilde r_{2,m}- r_{2,n}\|^2_{1}  ds,
\end{align*}
with constants $c,C$ for all $n,m$, only depending on the uniform bounds
\begin{equation*}
\|\phi_n\|_{L^\infty(H^1_p)}, \|\hat\phi_m\|_{L^\infty(H^1_p)}, \|\dt \hat\phi_m\|_{L^{1}(L^2)}.    
\end{equation*}
Hence the constants $c,C$ in the above estimate can be chosen independent of $m,n$. 
Using the strong convergence of the residuals in the corresponding norms, we may pass to the limit in the integral on the right hand side. 
Furthermore application of Egorov's Theorem yields almost everywhere convergence of $b(\phi_n)$. With this and standard weak convergence results using Fatou's Lemma yields
\begin{equation*}
  \int_0^t\D_{\phi}(\mu(s)|\hat\mu(s)) ds \leq\liminf_{m\to \infty}\liminf_{n\to \infty}\int_0^t\D_{\phi_n}(\mu_n(s)|\hat\mu_m(s)) ds,
\end{equation*}
which allows us to pass to the limit in the relative dissipation term. A similar process is used to derive the energy inequality for the standard weak solution of \eqref{eq:ch1}-\eqref{eq:ch2}.
By the continuous embedding of $W(0,T)$ into $C([0,T];H^1_p(\Omega)$, 
we obtain convergence of the energy $\E_\alpha(\phi_n(t)|\hat \phi_m(t)) \to \E(\phi|\hat \phi)$ with $m,n \to \infty$, and in summary, we thus obtain \eqref{eq:stability}.

\end{document}